\documentclass[reqno,12pt]{amsart}
\usepackage[a4paper, margin=20mm]{geometry}

\usepackage{amsmath,stmaryrd, mathtools}
\usepackage{amssymb}
\usepackage{amsthm}
\usepackage[utf8]{inputenc}
\usepackage{graphicx}
\usepackage{xcolor}
\usepackage[normalem]{ulem}
\usepackage[english]{babel}

\usepackage{bbm} 
\usepackage{mathrsfs} 
\usepackage{stackengine} 
\usepackage{array} 
\usepackage{bm} 

\usepackage{enumitem} 
\usepackage{hyperref} 
\newcommand{\sref}[2]{\hyperref[#2]{#1~\ref{#2}}} 

\usepackage[misc]{ifsym} 

\def\geq{\ensuremath\geqslant}
\def\leq{\ensuremath\leqslant}

\def\Z{\ensuremath\mathbb{Z}}
\def\Q{\ensuremath\mathbb{Q}}
\def\R{\ensuremath\mathbb{R}}
\def\C{\ensuremath\mathbb{C}}
\def\K{\ensuremath\mathbb{K}}

\def\den{\ensuremath\mathrm{den}}
\def\Li{\ensuremath\mathrm{Li}}

\def\Qbar{\ensuremath\overline{Q}}

\newcommand\house[1]{%
  \begingroup\setlength\arraycolsep{0pt}
  \begin{array}[t]{@{}c@{}|c|@{}c@{}}
  \firsthline
  &\;#1\;{}&
  \end{array}
  \endgroup
} 

\newtheorem{theorem}{Theorem}[subsection]

\newtheorem{lemma}[theorem]{Lemma}

\newtheorem{proposition}[theorem]{Proposition}
\newtheorem{notation}[theorem]{Notation}
\newtheorem{notations}[theorem]{Notations}

\newtheorem{remark}{Remark}[subsection] 
\newtheorem*{remark*}{Remark}

\makeatletter
\newcommand\footnoteref[1]{\protected@xdef\@thefnmark{\ref{#1}}\@footnotemark}
\makeatother

\usepackage{csquotes}

\allowdisplaybreaks

\numberwithin{equation}{section}

\title[]{Linear independence of values of polylogarithms with periodic coefficients, outside the disk of convergence}
\author[]{Ludovic Mistiaen}
\thanks{Ludovic Mistiaen (Corresponding author) - Institut Fourier, Université Grenoble Alpes, CNRS, CS
40700, 38058 Grenoble cedex 9, France.} 



\begin{document}
\maketitle

\begin{abstract}
For any non-zero periodic function $f: \Z \to \C$ of period $N\geq 1$ and any algebraic number $z_0 \neq 0$ that doesn't lie on a half-line $[e^{2i\ell\pi/N}, e^{2i\ell\pi/N}\infty[$, $0\leq\ell\leq N-1$, we give a lower bound of order $\sqrt{{s}/{\log(s)}}$ on the dimension of the $\Q(z_0)$-vector space spanned by the numbers $L(f, i, z_0) := \sum_{m=1}^{+\infty} f(m)z_0^m/m^i$, $1\leq i\leq s$. This generalizes a result Fischler proved in 2026, corresponding to $f$ identically equal to $1$ and $|z_0| \leq 1$: in this case, the numbers $L(f, i, z_0) = \Li_i(z_0)$ are polylogarithm values. Except for a finite number of cuts in the complex plane, our result still holds when $|z_0|>1$, that is in a domain where the series definition above for the numbers $L(f, i, z_0)$ doesn't converge anymore, and we make sense of these numbers through analytic continuation. 
To obtain this result, we construct linear combinations of the numbers $L(f, i, z_0)$ using a refined version of Siegel's lemma, and we apply to them a linear independence criterion generalizing the one used by Fischler. To check the assumptions of this criterion, we rely on an integral representation of the polylogarithm functions and on a ``Shidlovskii lemma".
\end{abstract}

\section{Introduction}
\setcounter{subsection}{1}

About the arithmetic nature of the values of the Riemann zeta function at integers, it is long known that the number $\zeta(2k)$ is a rational multiple of $\pi^{2k}$, hence transcendental. However, not much is known about the numbers $\zeta(2k+1)$, even though these numbers are conjectured to be
algebraically independent over $\mathbb{Q}$, together with $\pi$, see \cite[Conjecture 0.1]{fischler2004}. The first
result in this direction is due to Apéry \cite{apery79}, who proved that
$\zeta(3)$ is irrational. Ball and Rivoal \cite{rivoal2000, ballrivoal2001} then obtained a
lower bound on the dimension of the vector space spanned by these numbers.
 
\begin{theorem}
\label{thmballrivoal}
    (Ball-Rivoal, 2001).
    For odd integers $s\geq 3$, we have as $s$ goes to infinity
    \begin{equation*}
        \mathrm{dim}_\Q \mathrm{Span}_\Q \Big\{1, \zeta(3), \zeta(5), ..., \zeta(s)\Big\} \geq  \hspace{1mm}\frac{1+o(1)}{1+\log(2)}\log(s).
    \end{equation*}
\end{theorem}

Their proof consists in applying Nesterenko's linear independence
criterion \cite{nesterenko85} to a sequence of explicit linear combinations of
$1, \zeta(3), ..., \zeta(s)$, obtained by summing at positive integers the
values of a rational function. This mechanism goes back to
Nikishin \cite{nikishin79} and is closely related to Padé approximation, see
\cite{fischlerrivoal2003} for an overview. Nesterenko's criterion requires both an upper and
a lower bound on the size of the linear combinations, which are provided by the explicit construction. Counting irrational numbers instead of
bounding a dimension, Fischler, Sprang and Zudilin \cite{fischlersprangzudilin2019}, then Lai and
Yu \cite{laiyu2020, lai2025}, obtained lower bounds of respective orders
$2^{\frac{\log(s)}{\log\log(s)}(1+o(1))}$ and
$\big(1.28+o(1)\big)\sqrt{s/\log(s)}$ for the number of irrationals among
$\zeta(3), \zeta(5), \ldots, \zeta(s)$. Such statements are of a weaker nature than
the linear independence provided by \sref{Theorem}{thmballrivoal}.
 
\noindent Fischler \cite{fischler2026} recently improved the lower bound of \sref{Theorem}{thmballrivoal} by establishing
\begin{equation}
\dim_{\Q} \mathrm{Span}_{\Q}
\Big( 1, \zeta(3), \zeta(5), \ldots, \zeta(s) \Big)
\geqslant 0.21 \sqrt{\frac{s}{\log(s)}}
\end{equation}
for any sufficiently large odd integer $s$. The novelty of his approach is to
give up explicit rational functions, and to construct them instead by means of
Siegel's lemma: the resulting coefficients are much smaller than in explicit
constructions, which is what makes the gain from $\log(s)$ to
$\sqrt{s/\log(s)}$ possible. However, this approach doesn't provide a lower bound on the size of the linear combinations, and Nesterenko's criterion has to be replaced by a criterion
\emph{à la} Siegel, in which that lower bound is traded for the construction of
sufficiently many independent linear combinations. This last point is checked
by means of a ``Shidlovskii's lemma'', a multiplicity estimate for solutions of
differential systems introduced in \cite{shidlovskii89} and developed
in \cite{bertrandbeukers85, bertrand2012, fischler2018}.

\smallskip
Similar results exist when the numbers $\zeta(i)$ are replaced by the values
$L(f,i) \break := \sum_{m \geqslant 1} f(m)/m^i$ of the $L$-function attached to a
periodic function $f : \mathbb{Z} \to \mathbb{C}$ of period $N \geqslant 1$. The
case $f$ identically equal to $1$ is that of zeta values. Nishimoto \cite{nishimoto2011} generalized
\sref{Theorem}{thmballrivoal} to this setting by establishing, for any parity $\varepsilon$ and as the integer $s$ of parity $\varepsilon$ goes to infinity:
\begin{equation}
\label{thmnishimoto}
        \mathrm{dim}_\Q \mathrm{Span}_\Q \Big\{ \;L(f, i) \;|\; 2\leq i\leq s, \;i\equiv \varepsilon \mod 2 \; \Big\} \geq \frac{1+o(1)}{N+\log(2)}\log(s).
\end{equation}
Restricting the integers $i$ to a fixed parity class is necessary, since for an even, resp. odd, Dirichlet character $\chi$, the numbers $L(\chi,i)$ with $i$ even, resp. odd,
are in $\overline{\Q}\pi^i$, making the lower bound in \eqref{thmnishimoto} trivial (see for example \cite[Corollary 2.10]{neukirch97}). Fischler \cite{fischler2018}
recovered this result with the method of ``Shidlovskii's lemma'' used in the
present paper, and removed the dependence on $N$ in the constant
in \cite{fischler2020}. Finally, the author \cite{chapitre1} transposed the construction
of \cite{fischler2026} to this setting, and obtained, for a non-principal Dirichlet character $\chi$
of conductor $N \not\equiv 0 \mod 4$, and
$\mathbb{K} = \mathbb{Q}\big(e^{2i\pi/N}\big)$:
\begin{equation*}
        \mathrm{dim}_\K \mathrm{Span}_\K \Big\{ \;L(\chi, i) \;|\; 2\leq i\leq s, \;i\equiv \varepsilon \:[2] \; \Big\} \geq \frac{0.42}{\varphi(N)\sqrt{N}}\sqrt{\frac{s}{\log(s)}},
\end{equation*}
where $\varphi$ is Euler's totient function and $\varepsilon$ is the parity opposed to that of $\chi$.
 
\smallskip
In another direction, the numbers $\zeta(i)$ are the values at $z=1$ of the polylogarithm functions
$\mathrm{Li}_i(z) = \sum_{m \geqslant 1} z^m/m^i$, and they can be replaced by numbers $\mathrm{Li}_i(z_0)$ at an algebraic point $z_0$. When
$|z_0|$ is small, linear independence can be proved for every $s$:
Nikishin \cite{nikishin79} proved that the numbers
$1, \mathrm{Li}_1(p/q), \ldots, \mathrm{Li}_s(p/q)$ are
$\mathbb{Q}$-linearly independent as soon as $|p| \geqslant 1$ and
$|q| \geqslant (4s)^{s(s-1)}|p|^s$. This was later refined by Hata \cite{hata90}.
The Chudnovsky brothers \cite{chudnovsky84, chudnovskychudnovsky85} obtained analogous statements at an arbitrary algebraic
point $z_0$ whose modulus is small enough in terms of $s$, and of the degree
and the height of $z_0$. Their methods rely on Padé approximation and on the
theory of $G$-functions, of which the polylogarithms are a basic example.
 
\noindent When $z_0$ with $|z_0| \leq 1$ is arbitrarily fixed and $s \to +\infty$, the situation is exactly the one
described above for zeta values. Adapting the proof of \sref{Theorem}{thmballrivoal}, Rivoal
\cite{rivoal2003} for real points, and Marcovecchio \cite{marcovecchio2006} in general, proved
that for any non-zero algebraic $z_0$ with $|z_0| < 1$, we have as the integer $s$ goes to infinity
\begin{equation}
\dim_{\mathbb{Q}(z_0)} \mathrm{Span}_{\mathbb{Q}(z_0)}
\Big\{ \mathrm{Li}_i(z_0) \mid 1 \leqslant i \leqslant s \Big\}
\geqslant \frac{1+o(1)}{(1+\log(2))[\mathbb{Q}(z_0):\mathbb{Q}]} \log(s).
\end{equation}
Fischler's method yields the following improvement \cite[Theorem 2]{fischler2026}: for any algebraic $z_0$ with $|z_0| \leqslant 1$ and
$z_0 \notin \{0,1\}$, and any integer $s$ large enough, we have
\begin{equation}
\dim_{\mathbb{Q}(z_0)} \mathrm{Span}_{\mathbb{Q}(z_0)}
\Big\{ 1, \mathrm{Li}_i(z_0) \mid 1 \leqslant i \leqslant s \Big\}
\geqslant \frac{0.26}{[\mathbb{Q}(z_0):\mathbb{Q}]} \sqrt{\frac{s}{\log(s)}},
\end{equation}
the constant $0.26$ being replaced by $0.52$ when
$\mathbb{Q}(z_0) \not\subseteq \mathbb{R}$.
 
\smallskip
All the results quoted so far concern points $z_0$ lying in the closed unit
disk, that is, in the domain where the series defining $\mathrm{Li}_i(z_0)$ and
$L(f,i,z_0) := \sum_{m \geqslant 1} f(m)z_0^m/m^i$ converge. Outside that disk,
these numbers are still defined by analytic continuation, but the linear combinations are no longer given by
convergent series, while the estimates of their size usually rest on that
series representation. An integral representation can be used instead, as in \cite{fischlerrivoal2021}, which
treats the values of a general $G$-function $F$ at any algebraic point of a
domain $\mathcal{D}_F$ star-shaped at $0$ and containing the disk of
convergence, and obtains there a lower bound of order $\log(S)$, thereby
extending to that domain the results of \cite{fischlerrivoal2020}.
 
\smallskip
In this paper, we combine the non-explicit construction of \cite{fischler2026} with the
$L$-function setting of \cite{fischler2018, chapitre1}, and we use an integral
representation of the polylogarithm functions which allows the estimates
of \cite[Lemma 6]{fischler2026} to be carried outside the unit disk. Our main result is
the following. Here and throughout the paper, the numbers $L(f,i,z)$ are
defined on $\mathcal{D}$ by analytic continuation, as explained in \sref{Section}{partie2}.

\begin{theorem}
\label{théorèmeprincipal}
    Let $N\geq 1$ be an integer. Let $f: \Z \to \C$ be a non-zero $N$-periodic function. Let $\mu = e^{\frac{2i\pi}{N}}$ and $\mathcal{D}$ be the complex plane minus $0$ and the half-lines $[\mu^\ell, \mu^\ell \infty [$, $0\leq \ell \leq N-1$. \\ 
    Then for any algebraic number $z_0 \in\mathcal{D}$ and any integer $s$ large enough, we have
    \begin{equation*}
        \mathrm{dim}_{\Q(z_0)} \mathrm{Span}_{\Q(z_0)} \Big\{ \;L(f, i, z_0) \;|\; 1\leq i\leq s \Big\} \geq \frac{0.26}{[\Q(z_0):\Q]\sqrt{N}}\sqrt{\frac{s}{\log(s)}}.
    \end{equation*}
    Moreover, if $\Q(z_0) \not\subseteq \R$, the constant $0.26$ can be replaced by $0.52$.
    \end{theorem}

\begin{remark}
No restriction on the parity of $i$ is needed here, in contrast with the results
of~\cite{rivoal2000, nishimoto2011, fischler2018, fischler2020, chapitre1} quoted above: the special values which make such a restriction necessary occur at $z=1$ only. In particular, unlike in \cite{chapitre1}, we don't need to use the series $S_{n, p}^{[0]}$ and the polynomials $\overline{Q}_{n, 0, (p, k)}$ in \sref{Section}{partie4}. The series $S_{n, p}^{[\infty]}$ and the polynomials $Q_{n, 0, (p, k)}$ are enough.
\end{remark}

\section{Structure of the proof}
\setcounter{subsection}{1}
\label{partie2}

For any periodic function $f: \Z \to \C$, we introduce
\begin{equation*}
    L(f, i, z) := \sum_{m=1}^{+\infty} \frac{f(m)z^m}{m^i}, \;\;\;\;\;i\in\Z_{>0}, \;\;\;|z| < 1.
\end{equation*}
If $f$ is identically equal to $1$, this yields the polylogarithm functions
\begin{equation}
\label{définitiondespolylogarithmes}
    \Li_i(z) = \sum_{m=1}^{+\infty} \frac{z^m}{m^i}, \;\;\;\;\;i\in\Z_{>0}, \;\;\;|z| < 1.
\end{equation}
It is well known that the polylogarithm functions can be analytically continued to the simply connected domain $\C\setminus[1, +\infty[$, see for example \cite[§1.1]{oesterle93}. We will denote by $\Li_i(z)$ the holomorphic function on $\C\setminus[1, +\infty[$ that coincides with the series representation \eqref{définitiondespolylogarithmes} on the domain $|z| < 1$.

\begin{notations}
\label{notations}
    For the entire paper, we fix an integer $N \geq 1$ and a non-zero $N$-periodic function $f: \Z \to \C$. We let $\mu = e^{2i\pi/N}$. We denote by $\mathcal{D}$ the complex plane minus $0$ and the half-lines $[\mu^\ell, \mu^\ell \infty [$, $0\leq \ell \leq N-1$. 

    We also fix an embedding $\overline{\Q}\hookrightarrow\C$ and an algebraic point $z_1$ such that $z_0 := 1/z_1 \in\mathcal{D}$. We denote the denominator of $z_1$ by $\mathrm{den}(z_1)$. We denote by $\log$ a fixed determination of the complex logarithm on $\mathcal{U}$, where $\mathcal{U}$ is the complex plane minus an arbitrary half-line starting from $0$ and not going through $1$ nor any point $\mu^\ell z_1$, $0\leq\mu\leq N-1$.
\end{notations}

We will show in \sref{Lemma}{formulesd'inversiondeFourier} that the function $L(f, i, z)$ is a $\C$-linear combination of the functions $\Li_i(\mu^\ell z)$, $0\leq\ell\leq N-1$, on the domain $|z|<1$. Using the analytic continuation of the polylogarithms, we can thus make sense of $L(f, i, z)$ on $\mathcal{D}$. 

\begin{remark}
    Notice that if $\Q(z_0) \subset \R$, we have $$\dim_{\Q(z_0)} \mathrm{Span}_{\Q(z_0)} \Big\{ \;L(f, i, z_0) \;|\; 1\leq i\leq s \Big\} \geq \dim_{\Q(z_0)} \mathrm{Span}_{\Q(z_0)} \Big\{ \;L(\mathfrak{Re}f, i, z_0) \;|\; 1\leq i\leq s \Big\},$$ and the same holds upon replacing $\mathfrak{Re}$ by $\mathfrak{Im}$. In the case $z_0 \in \R$, we can thus assume that $f$ takes values in $\R$, and hence $L(f, i, z_0) \in \R$ for all $i\in\Z_{>0}$, without loss of generality.
\end{remark}

\smallskip
Let us now describe the structure of this paper, dedicated to the proof of \sref{Theorem}{théorèmeprincipal}. If $F$ is a rational function of degree $\leq -2$, whose poles are negative integers $\geq -r$ and of order $\leq a\in\Z_{>0}$, then the series
\begin{equation*}
    S^{[\infty]}(z) = \sum_{t=1}^{+\infty} F(t)z^{-t}
\end{equation*}
is a linear combination with polynomial coefficients of $1, \Li_1(1/z), ..., \Li_a(1/z)$ (see for example \cite[Lemme 1]{ballrivoal2001} or \cite[§4.3]{fischler2018}). If the rational function differentiated $p$ times $F^{(p)}$ is used instead, polylogarithms of order $p+1$ to $a+p$ appear instead (see \cite{nesterenko96} for the example of $\zeta(3)$, and \cite[§4]{fischlerrivoal2003} for the general case). This mechanism relies on the partial fraction decomposition of such a rational function, as it can be seen in the proof of \sref{Lemma}{lesSsontdesclenlespolylogs} below.

In \sref{Section}{partie3}, we recall the result of \cite[§3]{chapitre1}. We fix a positive integer $a\in\Z_{>0}$, we introduce rational parameters $r, \omega, \Omega$, and we obtain such a rational function $F_n$ for infinitely many integers $n$ (large enough\footnote{\label{footnote1} In a sense given by equation \eqref{nassezgrand2}.}, multiples of $N$, and such that $\frac{rn}{N}, \omega n$ and $\Omega n$ are integers). The construction of \cite[§3]{chapitre1} uses Siegel's lemma and is therefore not explicit, but the properties of these rational functions enable us to construct linear combinations to which the linear independence criterion given by \sref{Proposition}{critèred'indépendancelinéaire} will apply.

In \sref{Section}{partie4}, we introduce two new parameters $h\in\Z_{\geq 0}, \kappa\in\Q$. We denote by $\mathcal{N}$ the infinite set of integers large enough\footnoteref{footnote1}, multiples of $N$ and such that $\frac{rn}{N}, \omega n, \Omega n$ and $\kappa n$ are integer. For each $n\in\mathcal{N}$ and $p\in\llbracket 0, h\rrbracket$, we use $F_n^{(p)}$ to construct linear combinations $S_{n, p}^{[\infty]}(z)$ with coefficients in $\Q[z]$ of $1$ and $\Li_1(1/z), ..., \Li_{a+h}(1/z)$. We differentiate them $k-1$ times with respect to $z$ and we combine them in the fashion of \cite[§4.3]{fischler2018} to obtain linear combinations $\Lambda_{n, (p, k)}(z_1), \;(p, k)\in\llbracket 0, h\rrbracket\times\llbracket (r+1)n+1, \kappa n\rrbracket$ of the numbers $f(0), ..., f(N-1)$ and $L(f, i, 1/z_1), \;\;1\leq i\leq a+h$. 

In \sref{Section}{partie5}, we prove that the linear combinations $\Lambda_{n, (p, k)}(z_1)$ satisfy the assumptions $(i)$ and $(ii)$ of the linear independence criterion (\sref{Proposition}{critèred'indépendancelinéaire}), that is, their absolute value goes to $0$ at least geometrically fast with $n$, while their coefficients are algebraic integers in $\mathcal{O}_{\Q(z_1)}$ going to infinity at most geometrically fast with $n$.

In \sref{Section}{partie6}, we rely on a ``Shidlovskii's lemma" (recalled in \cite[Theorem 6.1.1]{chapitre1}) to prove that the linear combinations $\Lambda_{n, (p, k)}(z_1)$ satisfy, under some assumption on the parameters $\omega, r, h$ and $\kappa$, the assumption $(iii)$ of the linear independence criterion for some linear forms $\psi_\ell$. To this end, we will need to consider a differential system which has a singularity at every $N$-th root of unity.

In \sref{Section}{partie7}, we state a linear independence criterion (\sref{Proposition}{critèred'indépendancelinéaire}). We then fix some numerical values for the parameters $r, \omega, \kappa$, some values depending on $a$ for the parameters $\Omega$ and $h$. We then apply the linear independence criterion with one of the linear forms $\psi_\ell$, for a fixed $a$. Finally, we give the asymptotics with $a\to +\infty$ of the result we obtain.

\section{A sequence of non explicit rational functions}
\label{partie3}
\setcounter{subsection}{1}

Let $a\geq N+1$ be a fixed integer. \sref{Proposition}{existencedescij} below constructs a rational function $F_n\in\Q(t)$ for infinitely many multiples $n$ of $N$. The rational function $F_n$ is of degree $\leq -2$ and its poles are among $0, -N, -2N, ..., -n$ and of order at most $a$. It is thus characterized by the coefficients $c_{n, i, j}$ of its partial fraction decomposition, $(i, j)\in\llbracket 1, a\rrbracket\times\llbracket 0, n/N\rrbracket$. We also introduce the coefficients $\mathfrak{A}_{n,k}$ of its Taylor expansion at $+\infty$. These coefficients all depend on $a$ and on the parameters $\omega, \Omega$ and $r$ defined below, but we ignore this in the notation. We have
\begin{equation}
\label{DESdefn}
    F_n(t) = \sum_{i=1}^a\sum_{j=0}^{n/N} \frac{c_{n, i, j}}{(t+Nj)^i} = \sum_{k=1}^{+\infty} \frac{\mathfrak{A}_{n,k}}{t^k}.
\end{equation}

\begin{notation}
\label{notationdketdeltaak}
    We denote by $d_k$ the $\mathrm{lcm}$ of all integers from $1$ to $k$, and by $\Delta_{a, k}$ the $\mathrm{lcm}$ of all products of $a$ distinct integers taken in an interval $I\subseteq \llbracket -k, k\rrbracket$ of length at most $k$.

    According to \cite[Lemma 2]{fischler2026}, they satisfy as $k$ goes to infinity:
    \begin{equation}
    \label{estimationexplicitedketdeltaak}
        d_k \leq e^{k+o(k)}, \hspace{1.5cm} \Delta_{a, k} \leq (a+1)^{k}e^{\gamma k + o(k)}.
    \end{equation}
\end{notation}

The following proposition states the existence of rational functions that are suitable for our purposes in the following sections. Its proof is the same as \cite[Proposition 3.1.1]{chapitre1}, upon using the asymptotic estimations \eqref{estimationexplicitedketdeltaak} instead of the effective global estimations \cite[(3.7)]{chapitre1} where ``$k+o(k)$" is replaced by ``$1.04k$". This proof uses a version of Siegel's lemma as stated in \cite[Lemma 1]{fischler2026}.

\begin{proposition}
\label{existencedescij}
    Let $a\geq N+1$ be an integer. Let $\omega, \Omega, r \in\Q$ be parameters satisfying $1\leq\omega\leq\Omega<\frac{a}{N}$ and $r\geq 1$. Then for any $n\in\Z_{>0}$ such that $\omega n, \Omega n, \frac{n}{N}, \frac{rn}{N}$ are integers, there exist coefficients $c_{n, i, j}\in\Z$ not all zero such that
    \begin{small}
    \begin{enumerate}
        \item[$(i)$] $F_n(t) \stackunder{$=$}{$\scriptscriptstyle{t\to+\infty}$} \mathcal{O}(t^{-\omega n})$,
        \item[$(ii)$] $|\mathfrak{A}_{n,k}| \stackunder{$\leq$}{$\scriptscriptstyle{n\to+\infty}$} \begin{cases}  r^{k-\Omega n}n^k k^a \xi^{n+o(n)} &\text{if $\omega n \leq k < \Omega n$,} \\ n^k k^a \xi^{n+o(n)} &\text{if $k \geq \Omega n$,} \end{cases}$ where the sequences $o(n)$ do not depend on $k$.  
        \item[$(iii)$] $\forall i\in\llbracket 1, a\rrbracket \;\forall j\in\llbracket 0, n/N \rrbracket, \; |c_{n, i, j}| \stackunder{$\leq$}{$\scriptscriptstyle{n\to+\infty}$} \xi^{n+o(n)}$, where $$\xi: = \exp\Bigg(\frac{\omega\log(2)+2\omega^2 + \omega^2\log(a+1)+\frac{1}{2}\Omega^2\log(r)}{\frac{a}{N}-\omega}\Bigg).$$
        \end{enumerate}
        \end{small}
\end{proposition}

In the course of the proof of this proposition, a crucial use of the rational functions
\begin{equation}
\label{definitiondesPki}
    P_{n, i}(z): = \sum_{j=0}^{n/N} c_{n, i, j}z^{Nj}, \hspace{1cm}
    \begin{cases}
        P_{n, 1, i} = P_{n, i}, \\
        P_{n, k+1, i} = P_{n, k, i}' - \frac{1}{z}P_{n, k, i+1},
    \end{cases} \hspace{1cm} 1\leq i\leq a,
\end{equation}
with the convention $P_{n, k, a+1} = 0$, is made. The polynomials $P_{n, i}$ will naturally appear in the proof of \sref{Lemma}{lesSsontdesclenlespolylogs}. In \autoref{partie5}, we will also use the rational functions
\begin{equation}
\label{recurrencedesPk0}
    \begin{cases}
        P_{n, 1, 0} = 0, \\
        P_{n, k+1, 0} = P_{n, k, 0}' + \frac{1}{z(1-z)}P_{n, k, 1}.
    \end{cases}
\end{equation}

The motivation behind the recurrence relations \eqref{definitiondesPki} and \eqref{recurrencedesPk0} is that any solution $\prescript{t}{}{\big(}y_0(z),\break y_1(z), ..., y_a(z)\big)$ of the differential system $Y' = AY$, where $A\in M_{a+1}(\Q(z))$ is given by 
\begin{equation}
\label{systemediffdespolylogarithmes}
    A = \begin{bmatrix}
        0 &0 &0 &\hdots &0 &0 \\
        \frac{1}{z(1-z)} &0 &0 &\hdots &0 &0 \\
        0 &\frac{-1}{z} &0 &\hdots &0 &0 \\
        0 &0 &\frac{-1}{z} &\hdots &0 &0 \\
        \vdots &\vdots &\vdots &\ddots &\vdots &\vdots \\
        0 &0 &0 &\hdots &\frac{-1}{z} &0
    \end{bmatrix},
\end{equation}
satisfies
\begin{equation*}
    \Bigg(\frac{d}{dz}\Bigg)^{k-1} \sum_{i=0}^a P_{n, i}(z)y_i(z) = \sum_{i=0}^a P_{n, k, i}(z)y_i(z).
\end{equation*}
Notice that the vector $\prescript{t}{}{\big(}1, \Li_1(1/z), ..., Li_a(1/z)\big)$ is a solution of this differential system.

\smallskip
The following proposition corresponds to \cite[Proposition 3.2.2]{chapitre1} and \cite[Proposition 3.4.1]{chapitre1}. It will be used in \autoref{paragrapheestimationdescoefficients} below.

\begin{proposition}
\label{propositionexistenceetestimationdestheta}
    Let $a \geq N+1$ be an integer. Let $n\in\Z_{>0}$. There exist rational coefficients $\theta_{a, n, k, i, j, \ell} \in \Q$ and $\vartheta_{a, n, k, 0, j, \ell, t}\in\Q$ such that for any choice of numbers $c_{n, i, j}$, the rational functions $P_{n, k, i}$ defined by the recurrence \eqref{definitiondesPki} and \eqref{recurrencedesPk0} satisfy for $k\geq 1$:
    \begin{small}
    \begin{equation}
    \label{definitiondesthetakijij}
        z^{k-1}P_{n, k, i}(z) = \sum_{j=0}^{n/N} \Bigg( \sum_{\ell=0}^{a-i} \theta_{a, n, k, i, j, \ell}\; c_{n, i+\ell, j}\Bigg) z^{Nj}, \hspace{1.5cm} 1\leq i\leq a.
    \end{equation}
    \end{small}
    and
    \begin{align}
    \label{expressiondesP0enfonctiondestheta}
            z^{k-1}(1-z)^{k-1}P_{n, k, 0}(z) = \sum_{t=0}^{n+k-1} \Bigg(\sum_{\ell=0}^{a-1} \sum_{j=0}^{n/N} \vartheta_{a, n, k, 0, j, \ell, t} \;c_{n, 1+\ell, j} \Bigg) z^t.
    \end{align}
    These coefficients satisfy
    \begin{enumerate}
        \item[$(i)$] $|\theta_{a, n, k, i, j, \ell}|\leq k^a2^n(k-1)!$, \hspace{1.6cm} $(i_0)\;\;|\vartheta_{a, n, k, 0, j, \ell, t}| \leq k^{a+1}8^{\max(k, n)}(k-1)!,$
        \item[$(ii)$] $\frac{d_k\Delta_{a, \max(k, n)}}{(k-1)!}\theta_{a, n, k, i, j, \ell} \in \mathbb{Z}$, \hspace{1.7cm} $(ii_0)\;\;\frac{d_k^2\Delta_{a, \max(k, n)}}{(k-1)!}\vartheta_{a, n, k, 0, j, \ell, t}\in\mathbb{Z}.$
    \end{enumerate}
\end{proposition}

\section{Construction of the linear combinations}
\label{partie4}

We fix an integer $a \geq 3N$ and rational parameters $r, \omega$ and $\Omega$ satisfying $2 < r+1 < \omega \leq \Omega < \frac{a}{N}$.
We also fix an integer parameter $h\in\llbracket 0, a\rrbracket$ and a rational parameter $\kappa$ satisfying $r+1 < \kappa < \omega$. We denote by $\mathcal{N}$ the infinite set of integers $n\geq 2$ large enough (in a sense given by equation \eqref{nassezgrand2}) and such that $\omega n, \Omega n, \frac{n}{N}, \frac{rn}{N}$ and $\kappa n$ are all integers.

In this section, we fix $n\in\mathcal{N}$ and we construct complex numbers $\Lambda_{n, (p, k)}(z_1)$, ${0\leq p\leq h}$, $(r+1)n+1 \leq k \leq \kappa n$, that are linear combinations with coefficients in $\Q(z_1)$ of the numbers $f(0), ..., f(N-1), L(f, 1, 1/z_1), ..., L(f, a+h, 1/z_1)$, where $z_1$ is defined in \sref{Notations}{notations}. We will see in \autoref{partie5} that these linear combinations actually have coefficients in $\mathcal{O}_{\Q(z_1)}$.

\smallskip
In \autoref{soussection4.1}, we construct for $p\in\llbracket 0, h\rrbracket$ linear combinations $S_{n, p}^{[\infty]}(z)$ with coefficients in $\Q[z]$ of the functions $1$ and $\Li_i(1/z)$, $1\leq i\leq a+h$.

In \autoref{paragrapheprocessusdedérivation}, we differentiate them with respect to $z$ to obtain for $k\in\llbracket (r+1)n+1, \kappa n\rrbracket$ linear combinations $S_{n, p}^{[\infty](k-1)}(z)$ with coefficients $\Q(z)$ of the functions $1$ and $\Li_i(1/z)$, $1\leq i\leq a+h$.

In \autoref{soussection4.3}, we combine the quantities $S_{n, p}^{[\infty](k-1)}(\mu^\ell z)$, $0\leq\ell\leq N-1$, to obtain functions $\widetilde{\Lambda}_{n, (p, k)}(z)$, that are linear combinations with coefficients in $\Q(z)$ of the constant functions $f(0), ..., f(N-1)$ and the functions $L(f, i, 1/z)$, $1\leq i\leq a+h$.

In \autoref{paragraphenotationspourlasuite}, we evaluate at the algebraic point $z=z_1$ and we multiply by an arithmetical factor and a power of $\den(z_1)$ to obtain numbers $\Lambda_{n, (p, k)}(z_1)$ that are linear combinations with coefficients in $\mathcal{O}_{\Q(z_0)}$ of the numbers $f(0), ..., f(N-1), L(f, i, 1/z_1), \;\;1\leq i\leq a+h$. We then introduce a few notation to improve clarity for \sref{sections}{partie5},\sref{}{partie6} and \sref{}{partie7}.

\subsection{Linear combinations in polylogarithms}
\label{soussection4.1}

\smallskip
Let $F_n$ be the rational function whose coefficients $c_{n, i, j}$ of the partial fraction decomposition \eqref{DESdefn} are given by \sref{Proposition}{existencedescij}. We consider for $p\in\llbracket 0, h\rrbracket$ the following series. It converges for $|z|> 1$, since $\mathrm{deg}\big(F_n^{(p)}\big) \leq -\omega n -p \leq -2$:
\begin{equation}
\label{définitiondessériesS}
    S_{n, p}^{[\infty]}(z) := z^{rn}\sum_{t=rn+1}^{+\infty} F_n^{(p)}(t)z^{-t}, \hspace{1.5cm} 0\leq p\leq h.
\end{equation}


The following lemma extends \cite[Lemma 4]{fischler2026}. It shows that these series are linear combinations of $1$ and the polylogarithms, with coefficients in $\Q[z]$:

\begin{lemma}
\label{lesSsontdesclenlespolylogs}
    There exist polynomials with rational coefficients $V_{n, p}^{[\infty]}$ of degree at most $(r+1)n-1$ and $Q_{n, i, p}, \;\;1\leq i\leq a+h,$ of degree at most $(r+1)n$ such that
    \begin{equation}
         \label{sériescommedesclenlespolylogarithmesbonscoefficients}
            S_{n, p}^{[\infty]}(z) = V_{n, p}^{[\infty]}(z) + \sum_{i=1}^{a+h} Q_{n, i, p}(z)\Li_i(1/z), \hspace{1.6cm}|z| > 1.
    \end{equation}
\end{lemma}

    The proof is as follows. By differentiating $p$ times equation \eqref{DESdefn}, we get
    \begin{equation}
    \label{dérivéepièmdelafractionrationnelle}
        F_n^{(p)}(t) = \sum_{i=1}^a \sum_{j=0}^{n/N} \frac{c_{n, i, j}(-1)^p(i)_p}{(t+Nj)^{i+p}}.
    \end{equation}

    We compute for $|z| > 1$:
    \begin{align*}
        S_{n, p}^{[\infty]}(z) &= z^{rn}\sum_{t=rn+1}^{+\infty} \sum_{i=1}^a \sum_{j=0}^{n/N} \frac{c_{n, i, j}(-1)^p (i)_p}{(t+Nj)^{i+p}}z^{-t} \\
                        &= \sum_{i=1}^a (-1)^p(i)_pz^{rn} \sum_{j=0}^{n/N} c_{n, i, j} \sum_{t=rn+1}^{+\infty} \frac{z^{-t}}{(t+Nj)^{i+p}} \\
                        &= \sum_{i=1}^a (-1)^p(i)_pz^{rn} \sum_{j=0}^{n/N} c_{n, i, j}z^{Nj} \sum_{t=rn+Nj+1}^{+\infty} \frac{z^{-t}}{t^{i+p}} &(t \xleftarrow{} t + Nj) \\  
                        &= \sum_{i=1}^a (-1)^p (i)_pz^{rn} \sum_{j=0}^{n/N} c_{n, i, j} z^{Nj}  \Bigg[ \Li_{i+p}(1/z) - \sum_{t=1}^{rn+Nj} \frac{z^{-t}}{t^{i+p}} \Bigg] \\
                        &= V_{n, p}^{[\infty]}(z) + \sum_{i=1}^a (-1)^p(i)_pz^{rn}P_{n, i}(z)\Li_{i+p}(1/z),
    \end{align*}
    where the polynomials $P_{n, i}$ are defined by \eqref{definitiondesPki}. The substitution $t \xleftarrow[]{} rn + Nj - t$ yields
    \begin{equation*}
        V_{n, p}^{[\infty]}(z) :=  -\sum_{i=1}^a (-1)^p(i)_p \sum_{j=0}^{n/N} c_{n, i, j} \sum_{t=0}^{rn+Nj-1}\frac{z^{t}}{(rn+Nj-t)^{i+p}}.
    \end{equation*}
    
    It is convenient to introduce the polynomials
    \begin{equation}
    \label{définitiondesQip}
        Q_{n, i+p, p}(z) := (-1)^p(i)_pz^{rn}P_{n, i}(z), \;\;1\leq i\leq a, \;\;\;\;\;\;\;\;\;\;\;\; Q_{n, i, p} = 0, \;\; i\notin \llbracket p+1, a+p \rrbracket,
    \end{equation}
    so that the above result can be rewritten as
    \begin{equation*}
            S_{n, p}^{[\infty]}(z) = V_{n, p}^{[\infty]}(z) + \sum_{i=1}^{a+h} Q_{n, i, p}(z)\Li_{i}(1/z), \hspace{1.6cm} |z| > 1.
    \end{equation*}

\begin{notations}
This lemma establishes equality \eqref{sériescommedesclenlespolylogarithmesbonscoefficients} on the domain $|z| > 1$. However, the polylogarithm functions $\Li_i(z)$, $i\geq 1$, can be analytically continued to the whole domain $\C\setminus [1, +\infty[$ (see for example \cite[1.1]{oesterle93}). The right-hand side of \eqref{sériescommedesclenlespolylogarithmesbonscoefficients} thus defines a holomorphic function on $\C\setminus [0, 1]$, which coincides with the series $S_{n, p}^{[\infty]}$ on the domain $|z| > 1$. 

From now on, we denote this function by $S_{n, p}^{[\infty]}$ on the whole domain $\C \setminus [0, 1]$, keeping in mind that the series representation \eqref{définitiondessériesS} is only valid for $|z| > 1$.
\end{notations}

\subsection{Differentiation process}
\label{paragrapheprocessusdedérivation}
Following \cite[section 4.1]{fischler2026}, in order to produce additional linear combinations and apply \sref{Proposition}{critèred'indépendancelinéaire} to them, we consider derivatives of the functions $S_{n, p}^{[\infty]}(z)$. We choose orders of derivation large enough to kill the polynomial parts $V_{n, p}^{[\infty]}(z)$ whose coefficients are too large for the diophantine methods to come.

The following differentiated series for $(p, k)\in\llbracket 0, h\rrbracket \times \llbracket (r+1)n+1, \kappa n\rrbracket$ are, again, linear combinations of $1$ and the polylogarithms:
\begin{align}
    \label{sériesdérivéescommedesclenlespolylogarithmesbonscoefficients}
    S_{n, p}^{[\infty](k-1)}(z) = Q_{n, 0, (p, k)}(z) + \sum_{i=1}^{a+h} Q_{n, i, (p, k)}(z)\Li_i(1/z), \hspace{1cm}z\in\C \setminus [0, 1].
\end{align}
Indeed, the vector of functions $\prescript{t}{}{\Big(}1, \Li_1(1/z), \Li_2(1/z), ..., \Li_{a+h}(1/z)\Big)$ is a solution of the differential system $Y' = AY$, where $A$ is the same as in \eqref{systemediffdespolylogarithmes} with $a$ replaced by $a+h$, so that the rational functions $Q_{n, i, (p, k)}$ can be computed with the following recurrence relation and the convention $Q_{n, a+h+1, (p, k)} = 0$:
\begin{equation}
\label{récurrencevérifiéeparlesQipk}
    \begin{cases}
        Q_{n, 0, (p, k+1)} = Q'_{n, 0, (p, k)} + \frac{1}{z(1-z)} Q_{n, 1, (p, k)}, &Q_{n, 0, (p, 1)} = 0, \\
        Q_{n, i, (p, k+1)} = Q'_{n, i, (p, k)} - \frac{1}{z} Q_{n, i+1, (p, k)}, &Q_{n, i, (p, 1)} = Q_{n, i, p}, \hspace{1cm} 1\leq i\leq a+h.
        
    \end{cases}
\end{equation}

\smallskip
We only consider indices $k \geq (r+1)n+1$, so that the terms $V_{n, p}^{[\infty](k-1)}$ that are supposed to appear in \eqref{sériesdérivéescommedesclenlespolylogarithmesbonscoefficients} vanish. The polynomials $V_{n, p}^{[\infty]}$ are indeed of degree at most $(r+1)n-1$.

Besides that, the only singularities that could appear at $z=1$ in equation \eqref{sériesdérivéescommedesclenlespolylogarithmesbonscoefficients} are:
\begin{enumerate}
    \item[$\bullet$] a logarithmic singularity due to $\Li_1$ if $Q_{n, 1, (p, k)}(1) \neq 0$,
    \item[$\bullet$] a pole if $Q_{n, 0, (p, k)}$ has a pole at $1$.
\end{enumerate}
But we only consider indices $k \leq \kappa n$, so that the series representation of $S_{n, p}^{[\infty](k-1)}(z)$ converges on the circle $|z|=1$. Indeed, we have
\begin{equation*}
        S_{n, p}^{[\infty](k-1)}(z) = z^{rn}\sum_{t=rn+1}^{+\infty} F_n^{(p)}(t)(-1)^{k-1}(t-rn)_{k-1}z^{-t-k+1},
\end{equation*}
and the term $F_n^{(p)}(t)(t-rn)_{k-1}$ is a rational function of degree $-\omega n -p +k - 1 \leq -2$ in the variable $t$. Thus, for $k\leq\kappa n <\omega n$, none of these two types of singularity happens at $z=1$, because they cannot cancel each other out. In particular, $Q_{n, 1, (p, k)}(1) = 0$ and the only pole that $Q_{n, 0, (p, k)}$ could have is 0. Because of the relation \eqref{récurrencevérifiéeparlesQipk}, this pole has order at most $k-1$. Thus, we have
\begin{equation}
    \begin{split}
    \label{pasdedivergencepolynomialenilogarithmique}
        &\forall (p,k)\in\llbracket 0, h\rrbracket\times\llbracket 1,\kappa n\rrbracket, \;\;\;\;\;\;\;\;\;\;\;\;\;\begin{split}
            &\;z^{k-1}Q_{n, 0, (p, k)} \;\;\text{is a polynomial} \\
            &\text{ of degree at most $(r+1)n$,}
        \end{split} \\ \\
        &\forall (p,k)\in\llbracket 0, h\rrbracket\times\llbracket 1,\kappa n\rrbracket, \;\;\;\;\;\;\;\;\;\;\;\;\; Q_{n, 1, (p, k)}(1) = 0.
    \end{split}
\end{equation}


\subsection{Linear combinations of $L$ functions}
\label{soussection4.3}

For $1\leq i\leq a+h$, notice that the polynomials $Q_{n, i, p}$ defined by \eqref{définitiondesQip} are in $\Q[z^N]$, since $P_{n, i}\in\Q[z^N]$ and $N | rn$. The recurrence relation \eqref{récurrencevérifiéeparlesQipk} then implies $z^{k-1}Q_{n, i, (p, k)} \in \Q[z^N]$ for all $k \geq 1$. Recall that $\mu$ is a $N$-th root of unity, so that
\begin{equation}
\label{lesQnipksontdansQzN}
    \forall i\in\llbracket 1, a+h\rrbracket,\;\;\;\forall \ell\in\llbracket 0, N-1\rrbracket, \;\;\;\;\;\;\;\;  (\mu^\ell z)^{k-1}Q_{n, i, (p, k)}(\mu^\ell z) = z^{k-1}Q_{n, i, (p, k)}(z).
\end{equation}

\smallskip
For $i=0$, the polynomials $z^{k-1}Q_{n, 0, (p, k)}(z)$ are not, in general, in $\Q[z^N]$. This is why we introduce the decomposition
\begin{equation}
    \label{décompositiondesQ0}
    z^{k-1}Q_{n, 0, (p, k)}(z) =: \sum_{m=0}^{N-1} z^mQ_{n, 0, (p, k)}^{<m>}(z), \;\;\;\;\;\;\;\;\;\;Q_{n, 0, (p, k)}^{<m>}\in\Q[z^N].
\end{equation}

Consider the coefficients
\begin{equation*}
    \hat{{f}}(\ell) := \frac{1}{N}\sum_{m=0}^{N-1} {f}(m)\mu^{-\ell m}, \;\;\;\;\;\;\;\; 0 \leq \ell \leq N-1.
\end{equation*}
They allow to reconstruct the function $L({f}, i, z)$ from the functions $\Li_i(z), \Li_i(\mu z), \Li_i(\mu^2 z), ...,\break \Li_i(\mu^{N-1} z)$, as shown in the following lemma. We also state a polynomial identity there, that will be used in \autoref{paragraphestructureshidlovskii}.
\begin{lemma}
\label{formulesd'inversiondeFourier}
    (Fourier inversion formulas) \;\;The following formulas hold.
    \begin{enumerate}
        \item[$(i)$] $\sum_{\ell = 0}^{N-1} \hat{{f}}(\ell)\mu^{\ell m} = {f}(m),\hspace{3.7cm} m\in\Z$,
        \item[$(ii)$] $\sum_{\ell = 0}^{N-1} \hat{{f}}(\ell)\Li_i(\mu^\ell z) = L({f}, i, z), \hspace{2.3cm} i\in\Z_{>0}, \;\;\;z\in\mathcal{D}$.
    \end{enumerate}
    \smallskip
    Moreover, the polynomials
    \begin{equation*}
            Q_{n, 0, (p, k)}^{\{\ell\}}(z) := \mu^{\ell(k-1)}Q_{n, 0, (p, k)}(\mu^\ell z)
    \end{equation*}
    satisfy for any $z\in\C^*$
    \begin{enumerate}
        \item[$(iii.a)$] $\sum_{m=0}^{N-1} \mu^{\ell m}z^{m}Q_{n, 0, (p, k)}^{<m>}(z) = z^{k-1}Q_{n, 0, (p, k)}^{\{\ell\}}(z), \hspace{2.75cm}0\leq \ell\leq N-1$,
        \item[$(iii.b)$] $\frac{1}{N}\sum_{\ell=0}^{N-1}\mu^{-\ell m}z^{k-1}Q_{n, 0, (p, k)}^{\{\ell\}}(z) = z^mQ_{n, 0, (p, k)}^{<m>}(z),\hspace{2.1cm}0\leq m\leq N-1$.
    \end{enumerate}
\end{lemma}

\begin{proof}
    \begin{enumerate}
        \item[]
        \item[$(i)$] We compute for $m\in\Z$
            \begin{align*}
                \sum_{\ell=0}^{N-1} \hat{{f}}(\ell) \mu^{\ell m} &= \frac{1}{N} \sum_{\ell=0}^{N-1} \sum_{\widetilde{m}=0}^{N-1} {f}(\widetilde{m}) \mu^{-\ell(\widetilde{m}-m)} \\
                                                    &= {f}(m) + \frac{1}{N}\sum_{{\widetilde{m}=0}\atop{\widetilde{m} \not\equiv m \;[N]}}^{N-1} {f}(\widetilde{m}) \left( \sum_{\ell = 0}^{N-1} \mu^{-\ell(\widetilde{m}-m)} \right).
            \end{align*}        
            We obtain the result since for $\widetilde{m} \not\equiv m \;[N]$, the inner sum over $\ell$ vanishes.
        \item[$(ii)$]
            We use $(i)$ and compute for $|z| < 1$
            \begin{align*}
                \sum_{\ell=0}^{N-1} \hat{{f}}(\ell)\Li_i(\mu^\ell z) = \sum_{\ell=0}^{N-1} \hat{{f}}(\ell) \left(\sum_{m=1}^{+\infty} \frac{\mu^{\ell m} z^m}{m^i}\right) &= \sum_{m=1}^{+\infty} \frac{z^m}{m^i} \left(\sum_{\ell=0}^{N-1} \hat{{f}}(\ell) \mu^{\ell m}\right) \\
                                                                &= \sum_{m=1}^{+\infty} \frac{{f}(m)z^m}{m^i} \\
                                                                &= L(f, i, z).
            \end{align*}
            This equality of holomorphic functions holds on the domain $|z|<1$. By the isolated zeros principle, it holds on $\mathcal{D}$.
        \item[$(iii.a)$] It is enough to apply, for $\ell \in \llbracket 0, N-1\rrbracket$ the substitution $\;z \xleftarrow[]{} \mu^\ell z\;$ in definition \eqref{décompositiondesQ0}. Recall that for $m\in\llbracket 0, N-1\rrbracket$ we have $Q_{n, 0, (p, k)}^{<m>}(\mu^\ell z) = Q_{n, 0, (p, k)}^{<m>}(z)$, since $Q_{n, 0, (p, k)}^{<m>} \in \Q[z^{\pm N}]$.
        \item[$(iii.b)$] We use $(iii.a)$ and compute for $m \in \llbracket 0, N-1\rrbracket$:
            \begin{align*}
                \frac{1}{N}\sum_{\ell=0}^{N-1} \mu^{-\ell m} &z^{k-1}Q_{n, 0, (p, k)}^{\{\ell\}}(z) = \frac{1}{N}\sum_{\ell=0}^{N-1} \sum_{\widetilde{m} = 0}^{N-1} \mu^{\ell(\widetilde{m}-m)} z^{\widetilde{m}} Q_{n, 0, (p, k)}^{<\widetilde{m}>}(z) \\
                &\hspace{1cm}= z^{m}Q_{n, 0, (p, k)}^{<m>}(z) + \frac{1}{N}\sum_{{\widetilde{m}=0}\atop{\widetilde{m} \neq m}}^{N-1} z^{\widetilde{m}} Q_{n, 0, (p, k)}^{<\widetilde{m}>}(z) \left( \sum_{\ell = 0}^{N-1} \mu^{-\ell(\widetilde{m}-m)} \right).
            \end{align*}
            We obtain the result since for $\widetilde{m} \neq m$, the inner sum over $\ell$ vanishes.
    \end{enumerate}
\end{proof}

We define, for $(p, k)\in\llbracket 0, h\rrbracket \times \llbracket (r+1)n+1, \kappa n\rrbracket$ and $z$ such that $1/z\in\mathcal{D}$:
\begin{equation}
\label{définitiondeLambdapk}
    \widetilde{\Lambda}_{n, (p, k)}(z) := \sum_{\ell = 0}^{N-1} \hat{{f}}(\ell) (\mu^{-\ell} z)^{k-1}S_{n, p}^{[\infty](k-1)}(\mu^{-\ell} z). \\
\end{equation}
This is a linear combination with coefficients in $\Q(z)$ of the constant functions ${f}(m)$ and the functions $L({f}, i, 1/z)$:

\begin{lemma}
\label{lemmeexpressiondelambdaenfonctiondez}
    Let $(p, k)\in\llbracket 0, h\rrbracket\times\llbracket (r+1)n+1, \kappa n\rrbracket$. The following equation holds for every $z\in\mathcal{D}^{-1}$:
    \begin{align}
        \begin{split}
        \label{expressiondelambdaenfonctiondez}
            \widetilde{\Lambda}_{n, (p, k)}(z) = \sum_{m=0}^{N-1} \Bigg(z^m Q_{n, 0, (p, k)}^{<m>}(z)\Bigg) {f}(-m) + \sum_{i=1}^{a+h} \Bigg(z^{k-1}Q_{n, i, (p, k)}(z) \Bigg) L({f}, i, 1/z).
        \end{split}
    \end{align}
\end{lemma}

\begin{proof}
    We compute
    \begin{align*}
        \widetilde{\Lambda}_{n, (p, k)}(z) &= \sum_{\ell = 0}^{N-1} \hat{{f}}(\ell) \Bigg(\sum_{m=0}^{N-1} (\mu^{-\ell} z)^m Q_{n, 0, (p, k)}^{<m>}(\mu^{-\ell} z)\Bigg) \\
        &\hspace{0.5cm} + \sum_{\ell = 0}^{N-1} \hat{{f}}(\ell) \Bigg( \sum_{i = 1}^{a+h} (\mu^{-\ell} z)^{k-1}Q_{n, i, (p, k)}(\mu^{-\ell} z)\Li_i\left(\frac{\mu^\ell}{z}\right)\Bigg)  \hspace{1cm}\text{\Big(from \eqref{sériesdérivéescommedesclenlespolylogarithmesbonscoefficients} and \eqref{décompositiondesQ0}\Big)} \\ 
                         &= \sum_{m=0}^{N-1} \Bigg[ z^m Q_{n, 0, (p, k)}^{<m>}(z) \Bigg(\sum_{\ell = 0}^{N-1} \hat{{f}}(\ell)\mu^{-\ell m} \Bigg) \Bigg] \\
                         &\hspace{0.5cm}+ \sum_{i=1}^{a+h} \Bigg[   z^{k-1}Q_{n, i, (p, k)}(z)  \Bigg( \sum_{\ell = 0}^{N-1} \hat{{f}}(\ell) \Li_i\left(\frac{\mu^\ell}{z}\right)\Bigg) \Bigg] \hspace{1.5cm}\text{\Big(from \eqref{lesQnipksontdansQzN} and \eqref{décompositiondesQ0}\Big)} \\
                         &= \sum_{m=0}^{N-1} \Bigg( z^m\Qbar_{n, 0, (p, k)}^{<m>}(z)\Bigg) {f}(-m) \\
                         &\hspace{0.5cm} + \sum_{i=1}^{a+h} \Bigg(  z^{k-1}Q_{n, i, (p, k)}(z) \Bigg) L({f}, i, 1/z) . \hspace{4cm}\text{\Big(\sref{Lemma}{formulesd'inversiondeFourier}\Big)}
    \end{align*}
\end{proof}

\subsection{Evaluation at an algebraic point $z_1$}
\label{paragraphenotationspourlasuite}

We define
\begin{equation*}
    \delta_{n, k} := d_k^2\Delta_{a+h, \max\big(k, (r+1)n\big)},
\end{equation*}
where $d_k$ and $\Delta_{a, k}$ are defined in \sref{notations}{notationdketdeltaak}. The quantity $\delta_{n , k}$ depends on the parameters $a+h \leq 2a$ and $r$. We are interested in its asymptotics as $n\to+\infty$, with $k$ bounded by $\kappa n$. Since $\kappa > r+1$, we have $\max\big(k, (r+1)n\big) \leq \kappa n$ and \eqref{estimationexplicitedketdeltaak} yields
\begin{equation}
\label{asymptotiquededelta}
    \delta_{n, k} \leq \Big( e^3 (2a+1) \Big)^{\kappa n}.
\end{equation}

\smallskip
Recall that the algebraic point $z_1 = 1/z_0$ was defined in \sref{Notations}{notations}. We consider the quantities 

\begin{equation}
\label{definitiondelambdafinal}
    \Lambda_{n, (p, k)}(z_1) := \frac{\delta_{n, k}}{(k-1)!} \den(z_1)^{(r+1)n} \widetilde{\Lambda}_{n, (p, k)}(z_1), \hspace{1cm} (p, k)\in\llbracket 0, h\rrbracket\times\llbracket (r+1)n+1, \kappa n\rrbracket.
\end{equation}

According to \sref{Lemma}{lemmeexpressiondelambdaenfonctiondez}, these are linear combinations of the numbers ${f}(0), ..., {f}(N-1)$, $L({f}, 1, 1/z_1), ...,$ $L({f}, a+h, 1/z_1)$ with coefficients in $\Q(z_1)\;$ (we will show in \autoref{paragrapheestimationdescoefficients} that the coefficients are actually in $\mathcal{O}_{\Q(z_1)}$). Let us define
\begin{equation}
\begin{split}
\label{définitiondeszetai}
    &\zeta_i := L(f, i, 1/z_1),\hspace{1cm} 1\leq i\leq a+h, \\
    &\zeta_0^{<m>} := f(-m), \hspace{2.3cm} 0\leq m\leq N-1.
    \end{split}
\end{equation}
\sref{Lemma}{lemmeexpressiondelambdaenfonctiondez} can then be rewritten as
\begin{equation}
\label{Lambdacommecldeszeta}
    \Lambda_{n, (p, k)}(z_1) = \sum_{m=0}^{N-1}\lambda_{n, 0, (p, k)}^{<m>}(z_1) \zeta_0^{<m>} + \sum_{i=1}^{a+h} \lambda_{n, i, (p, k)}(z_1) \zeta_i,
\end{equation}
with
\begin{equation}
\label{definitiondespetitslambdas}
    \begin{split}
        \lambda_{n, i, (p, k)}(z_1) & := \frac{\delta_{n, k}}{(k-1)!}\den(z_1)^{(r+1)n} z_1^{k-1}Q_{n, i, (p, k)}(z_1), \\
    \lambda_{n, 0, (p, k)}^{<m>}(z_1) & := \frac{\delta_{n, k}}{(k-1)!} \den(z_1)^{(r+1)n} z_1^m Q_{n, 0, (p, k)}^{<m>}(z_1).
    \end{split}
\end{equation}

\section{Upper bounds on the combinations and their coefficients}
\label{partie5}

To apply the linear independence criterion (\sref{Proposition}{critèred'indépendancelinéaire}) to the linear combinations $\Lambda_{n, (p, k)}(z_1)$ defined in \autoref{paragraphenotationspourlasuite}, we show in \autoref{paragrapheestimationdescoefficients} that they have coefficients in $\mathcal{O}_{\Q(z_1)}$, growing at most geometrically with $n$. We then show in \autoref{paragrapheestimationdelacombinaison} that $|\Lambda_{n, (p, k)}(z_1)|$ decreases at least geometrically fast to $0$ as $n\to+\infty$. For the latter, we use an integral representation that allows to generalize the bound of \cite[Lemma 6]{fischler2026} outside the region $|1/z| < 1$ where the series representation \eqref{définitiondessériesS} converges.

\subsection{Upper bound on the coefficients}
\label{paragrapheestimationdescoefficients}

Recall that $\xi$ was defined in \sref{Proposition}{existencedescij} and $z_1=1/z_0$ in \sref{Notation}{notations}.

\begin{proposition}
\label{majorationdescoefficients}
    For every $i\in\llbracket 1, a+h\rrbracket$, every $m\in\llbracket 0, N-1\rrbracket$ and every $(p, k)\in\llbracket 0, h\rrbracket \times \llbracket (r+1)n+1, \kappa n\rrbracket$, we have
    \begin{enumerate}[left=9pt]
        \item[$(i)$] $\house{\lambda_{n, i, (p, k)}(z_1)},\; \house{\lambda_{n, 0, (p, k)}^{<m>}(z_1)} \stackunder{$\leq$}{$\scriptscriptstyle{n\to+\infty}$} \beta^{n+o(n)}$, with $$\beta := \Big(\den(z_1) \max(1, \house{z_1})\Big)^{r+1} \Big(32e^3(2a+1)\Big)^\kappa \xi,$$
        \item[$(ii)$] $\lambda_{n, i, (p, k)}(z_1),\; \lambda_{n, 0, (p, k)}^{<m>}(z_1) \in \mathcal{O}_{\Q(z_1)}$.
    \end{enumerate}
\end{proposition}

\begin{proof}
    We fix $p\in\llbracket 0, h\rrbracket$. We write $Q_{n, i, p}(z) =: \sum_{j=0}^{(r+1)n/N} \gamma_{n, i, j}z^{Nj}$ for $i\geq 1$. From the definition \eqref{définitiondesQip} we deduce
    \begin{equation}
        \begin{split}
        \label{majorationdesgammaenfonctiondesc}
            &\bullet \;\; \gamma_{n, i, j} \in \Z, \\
            &\bullet \;\; \max_{{1\leq i \leq a+h}\atop{0\leq j\leq (r+1)n/N}} |\gamma_{n, i, j}| \leq (a+h)_p \max_{{1\leq i\leq a}\atop{0\leq j\leq n/N}} |c_{n, i, j}|.
        \end{split}
    \end{equation}
    
    The rational functions $Q_{n, i, (p, k)}$ and $Q_{n, 0, (p, k)}$ satisfy the recurrence relations \eqref{récurrencevérifiéeparlesQipk}, which are the exact same as equations $\eqref{definitiondesPki}$ and \eqref{recurrencedesPk0}. Therefore, \sref{Proposition}{propositionexistenceetestimationdestheta} yields
    \begin{align*}
        z^{k-1}Q_{n, i, (p, k)}(z) &= \sum_{j=0}^{(r+1)n/N} \Bigg( \sum_{\ell=0}^{a+h-i} \theta_{a+h, (r+1)n, k, i, j, \ell} \; \gamma_{n, i+\ell, j} \Bigg) z^{Nj}, \\
        z^{k-1}(1-z)^{k-1}Q_{n, 0, (p, k)}(z) &= \sum_{t=0}^{(r+1)n+k-1} \Bigg( \sum_{\ell=0}^{a+h-1} \sum_{j=0}^{(r+1)n/N} \vartheta_{a+h, (r+1)n, k, 0, j, \ell, t} \; \gamma_{n, 1+\ell, j} \Bigg) z^t, 
    \end{align*}
    with
    \begin{align}
    \begin{split}
    \label{majorationtheta}
        &|\theta_{a+h, (r+1)n, k, i, j, \ell}| \leq k^{a+h} 2^{(r+1)n} (k-1)!, \\
        &|\vartheta_{a+h, (r+1)n, k, 0, j, \ell, t}| \leq k^{a+h+1}8^{\max\big(k, (r+1)n\big)}(k-1)!,
    \end{split} \\
    \begin{split}
    \label{denomtheta}
        &\frac{\delta_{n, k}}{(k-1)!}\theta_{a+h, (r+1)n, k, i, j, \ell} \in \mathbb{Z}, \\
        &\frac{\delta_{n, k}}{(k-1)!}\vartheta_{a+h, (r+1)n, k, 0, j, \ell, t} \in \mathbb{Z}.
    \end{split}
    \end{align}

    For $k$ bounded by $\kappa n$, we deduce from \eqref{majorationdesgammaenfonctiondesc} and \eqref{majorationtheta} the asymptotic upper bound
    \begin{small}
    \begin{align*}
        \house{\lambda_{n, i, (p, k)}(z_1)} &=  \frac{\delta_{n, k}}{(k-1)!} \den(z_1)^{(r+1)n} \house{z_1^{k-1}Q_{n, i, (p, k)}(z_1)} \hspace{4.5cm}\text{\big(from \eqref{definitiondespetitslambdas}\big)}\\
        &\leq \frac{\delta_{n, k}}{(k-1)!}\den(z_1)^{(r+1)n} \Bigg(\sum_{j=0}^{(r+1)n/N} \house{z_1}^{Nj}  \sum_{\ell=0}^{a+h-i} |\theta_{a+h, (r+1)n, k, i, j, \ell}| \Bigg) \max_{i, j} |\gamma_{n, i, j}| \\
        &\leq \frac{\delta_{n, k}}{(k-1)!}\den(z_1)^{(r+1)n}\max\Big(1, \house{z_1}\Big)^{(r+1)n}\Big((r+1)\frac{n}{N} + 1\Big)(a+h)k^{a+h}2^{(r+1)n}(k-1)! \\
        &\hspace{11cm} \times (a+h)_h \max_{i, j}|c_{n, i, j}| \\
        &\stackunder{$\leq$}{$\scriptscriptstyle{n\to+\infty}$}  \Big( e^3(2a+1)\Big)^{\kappa n} \den(z_1)^{(r+1)n}\max\Big(1, \house{z_1}\Big)^{(r+1)n} \Big((r+1)\frac{n}{N} + 1\Big)(a+h)k^{a+h} 2^{\kappa n} \\
        &\hspace{2cm} \times (a+h)_h  \xi^{n+o(n)} \hspace{1.5cm} \big(\text{from \eqref{asymptotiquededelta}, $r+1 < \kappa$ and \sref{Proposition}{existencedescij}}\big) \\
        &\stackunder{$\leq$}{$\scriptscriptstyle{n\to+\infty}$} \Bigg(\Big(\den(z_1)\max\Big(1, \house{z_1}\Big)\Big)^{r+1}\Big(2e^3(2a+1)\Big)^\kappa  \xi \Bigg)^{n+o(n)}.
    \end{align*}
    \end{small}

    Then, we compute the coefficients of the polynomial $z^{k-1}Q_{n, 0, (p, k)}(z)$ as the Cauchy product of $z^{k-1}(1-z)^{k-1}Q_{n, 0, (p, k)}(z)$ and $\frac{1}{(1-z)^{k-1}} = \sum_{t=0}^{+\infty} \binom{t+k-2}{t} z^t$. For $q\in\llbracket 0, (r+1)n\rrbracket$, the coefficient of $z^q$ in $z^{k-1}Q_{n, 0, (p, k)}(z)$ is
    \begin{equation}
    \label{coefficientsdupolynomezk-1Q0}
        \sum_{t=0}^q \binom{q-t+k-2}{q-t} \Bigg( \sum_{\ell=0}^{a+h-1} \sum_{j=0}^{(r+1)n/N} \vartheta_{a+h, (r+1)n, k, 0, j, \ell, t} \; \gamma_{n, 1+\ell, j} \Bigg).
    \end{equation}

    Using \eqref{majorationdesgammaenfonctiondesc} and \eqref{denomtheta}, it follows that $\frac{\delta_{n, k}}{(k-1)!}z^{k-1}Q_{n, i, (p, k)}(z)$ and  $\frac{\delta_{n, k}}{(k-1)!}z^{k-1}Q_{n, 0, (p, k)}(z)$ are polynomials in $\Z[z]$. They are of degree at most $(r+1)n$, thus when we evaluate them at $z=z_1$ and multiply by $\den(z_1)^{(r+1)n}$, we obtain quantities in $\mathcal{O}_{\Q(z_1)}$. We recall that $\lambda_{n, i, (p, k)}(z_1)$ and $\lambda_{n, 0, (p, k)}^{<m>}(z_1)$ are defined in \eqref{definitiondespetitslambdas}, hence $(ii)$.
    
    \smallskip
    Finally, by definition \eqref{décompositiondesQ0}, the polynomial $z^m Q_{n, 0, (p, k)}^{<m>}(z)$ has, for every $m\in\llbracket 0, N-1\rrbracket$, at most $\frac{(r+1)n}{N}+1$ non-zero coefficients, that are all of the form \eqref{coefficientsdupolynomezk-1Q0}, so that
    \begin{small}
    \begin{align*}
        \house{\lambda_{n, 0, (p, k)}^{<m>}(z_1)} &= \frac{\delta_{n, k}}{(k-1)!} \den(z_1)^{(r+1)n} \house{z_1^m Q_{n, 0, (p, k)}^{<m>}(z_1)} \hspace{4cm}\text{\big(from \eqref{definitiondespetitslambdas}\big)} \\
        &\leq \frac{\delta_{n, k}}{(k-1)!}\den(z_1)^{(r+1)n}\max\Big(1, \house{z_1}\Big)^{(r+1)n} \Big(\frac{(r+1)n}{N}+1\Big) \\
        &\hspace{1cm} \times \sum_{t=0}^{(r+1)n} \Bigg( 2^{(r+1)n+k - 2}\bigg( \sum_{\ell=0}^{a+h-1} \sum_{j=0}^{(r+1)n/N} |\vartheta_{a+h, (r+1)n, k, 0, j, \ell, t}| \bigg)\Bigg) \max_{i, j}|\gamma_{n, i, j}| \\
        &\leq \frac{\delta_{n, k}}{(k-1)!}\bigg(\den(z_1) \max\Big(1, \house{z_1}\Big)\bigg)^{(r+1)n} \Big(\frac{(r+1)n}{N}+1\Big) \Big((r+1)n+1\Big) 2^{(r+1)n+\kappa n-2}\\
        &\hspace{1cm} \times (a+h)\Big(\frac{(r+1)n}{N}+1\Big) \times k^{a+h+1}8^{\max\big(k, (r+1)n\big)}(k-1)!(a+h)_h \max|c_{n, i, j}| \\
        &\stackunder{$\leq$}{$\scriptscriptstyle{n\to+\infty}$} \Big(e^3(2a+1)\Big)^{\kappa n} \bigg(\den(z_1) \max\Big(1, \house{z_1}\Big)\bigg)^{(r+1)n} \Big((r+1)n+1\Big)^3 2^{2\kappa n}\\
        & \hspace{0.2cm}\times  (a+h)k^{a+h+1}  8^{\kappa n} (a+h)_h \xi^{n+o(n)} \hspace{0.3cm} \big(\text{from \eqref{asymptotiquededelta}, $r+1 < \kappa$ and \sref{Proposition}{existencedescij}}\big) \\
        &\stackunder{$\leq$}{$\scriptscriptstyle{n\to+\infty}$} \Bigg(\bigg(\den(z_1) \max\Big(1, \house{z_1}\Big)\bigg)^{r+1} \Big(32e^3(2a+1)\Big)^\kappa \xi \Bigg)^{n+o(n)}.
    \end{align*}
    \end{small}
\end{proof}


\subsection{Upper bound on the combinations}
\label{paragrapheestimationdelacombinaison}

Recall that $\xi$ was defined in \sref{Proposition}{existencedescij} and $z_1=1/z_0$ in \sref{Notation}{notations}. We also use Pochhammer's symbol
\begin{equation}
\label{Pochhammer}
    (x)_j: = x(x+1)...(x+j-1), \;\;\;\;\;x\in\R, \;\;j\in\Z_{\geq 0}.
\end{equation}
The properties $\;\;\left(\frac{d}{dz}\right)^j z^x = (x-j+1)_jz^{x-j}$, $\;\;\binom{k}{j} = \frac{(k-j+1)_j}{j!}$ and $\;\;\binom{k}{j} \leq 2^k$ will be used without further reference.

\begin{proposition}
\label{majorationdescombinaisons}
    For every $(p, k)\in\llbracket 0, h\rrbracket \times \llbracket (r+1)n+1, \kappa n\rrbracket$, we have
    \begin{equation*}
        \left| \Lambda_{n, (p, k)}(z_1)\right| \leq \alpha^{n+o(n)},
    \end{equation*}
    with
    \begin{equation*}
        \alpha :=  \den(z_1)^{r+1}\Big(e^3(2a+1) \max(1, |z_1|)g(z_1)\Big)^\kappa \frac{\xi}{r^\Omega},
    \end{equation*}
    where $$g(z) := \max\bigg(1, \;\;\max_{v\in [0, 1], \ell\in\llbracket 0, N-1\rrbracket} \frac{1}{|z-v\mu^\ell|}\bigg).$$
    
\end{proposition}

\begin{proof}
    We write for $|z| > 1$ (which allows us to swap the series and the integral),

\begin{small}
\begin{align*}
    \frac{\delta_{n, k}}{(k-1)!}S_{n, p}^{[\infty](k-1)}(z) &= \frac{\delta_{n, k}}{(k-1)!} \sum_{t=rn+1}^{+\infty} F_n^{(p)}(t) (rn-t-k+2)_{k-1} z^{rn-t-k+1} \\
    &= (-1)^{k-1}\delta_{n, k}\sum_{t = rn+1}^{+\infty} \binom{t-rn+k-2}{k-1} F_n^{(p)}(t)z^{rn-t-k+1} \\
    &= (-1)^{k-1}\delta_{n, k} \sum_{t=k}^{+\infty} \binom{t-1}{k-1} F_n^{(p)}\big(t+rn-k+1\big)z^{-t} \hspace{1cm}(t\xleftarrow[]{} t-rn+k-1) \\
    &= (-1)^{k-1}\delta_{n, k}\sum_{t=k}^{+\infty} \binom{t-1}{k-1} \sum_{q=\omega n}^{+\infty} \frac{(-1)^p(q)_p \mathfrak{A}_{n, q}}{(t+rn-k+1)^{q+p}} z^{-t} \hspace{0.3cm} \big(\text{from \eqref{DESdefn} differentiated} \\ 
    &\hspace{8.3cm} \text{$p$ times and \sref{Proposition}{existencedescij} (i)}\big) \\
    &= (-1)^{p+k-1} \delta_{n, k} \sum_{q=\omega n}^{+\infty} (q)_p \mathfrak{A}_{n, q} \left( \sum_{t=k}^{+\infty} \binom{t-1}{k-1} \frac{z^{-t}}{(t+rn-k+1)^{q+p}}\right).
\end{align*}
\end{small}

Then, using the formulas $\frac{1}{t^q} = \frac{1}{\Gamma(q)}\int_0^{+\infty} u^{q-1}e^{-tu}\:du$ and $\sum_{t=k}^{+\infty} \binom{t-1}{k-1} z^{-t} = \frac{1}{(z-1)^k}$, we obtain the following integral representation :

\begin{align*}
    \sum_{t=k}^{+\infty} \binom{t-1}{k-1} \frac{z^{-t}}{(t+rn-k+1)^{q+p}} &= \sum_{t=k}^{+\infty} \binom{t-1}{k-1} \frac{z^{-t}}{\Gamma(q+p)}\int_0^{+\infty} u^{q+p-1}e^{-(t+rn-k+1)u}\:du \\
    &= \frac{1}{\Gamma(q+p)} \int_0^{+\infty} u^{q+p-1} e^{-(rn-k+1)u}\left( \sum_{t=k}^{+\infty} \binom{t-1}{k-1} (ze^u)^{-t}\right)\:du \\
    &= \frac{1}{\Gamma(q+p)}\int_0^{+\infty} \frac{u^{q+p-1}e^{-(rn-k+1)u}}{(ze^u - 1)^k}\:du \\
    &= \frac{1}{\Gamma(q+p)}\int_0^{+\infty} \frac{u^{q+p-1}e^{-(rn+1)u}}{(z - e^{-u})^k}\:du.
\end{align*}

We thus have
\begin{equation*}
    \frac{\delta_{n, k}}{(k-1)!}S_{n, p}^{[\infty](k-1)}(z) = (-1)^{p+k-1}\delta_{n, k} \sum_{q=\omega n}^{+\infty} \frac{(q)_p \mathfrak{A}_{n, q}}{\Gamma(q+p)}\int_0^{+\infty} \frac{u^{q+p-1}e^{-(rn+1)u}}{(z-e^{-u})^k}\:du.
\end{equation*}
The right-hand side defines a holomorphic function on $\C\setminus [0, 1]$, which coincides with the function $\frac{\delta_{n, k}}{(k-1)!}S_{n, p}^{[\infty](k-1)}(z)$ on the domain $|z| > 1$. By the isolated zeroes principle, both of these functions coincide on $\C\setminus [0, 1]$. This yields the following upper bound, on all of $\C\setminus [0, 1]$:
\begin{align}
\label{majorationdedeltafoisSsurkfact}
    \left|\frac{\delta_{n, k}}{(k-1)!}S_{n, p}^{[\infty](k-1)}(z)\right| &\leq \delta_{n, k} \sum_{q=\omega n}^{+\infty} \frac{(q)_p |\mathfrak{A}_{n, q}|}{\Gamma(q+p)}\int_0^{+\infty} \frac{u^{q+p-1}e^{-(rn+1)u}}{|z-e^{-u}|^k}\:du \nonumber \\
    &\leq \delta_{n, k} g(z)^k \xi^{n+o(n)} \left( \sum_{q=\omega n}^{\Omega n-1} \frac{(q)_p n^q q^a r^{q-\Omega n}}{(rn+1)^{q+p}} + \sum_{q=\Omega n}^{+\infty} \frac{(q)_p n^q q^a}{(rn+1)^{q+p}} \right) \nonumber \\
    &\leq \delta_{n, k} g(z)^k \xi^{n+o(n)} \left( \sum_{q=\omega n}^{\Omega n-1} \frac{(q)_p n^q q^a r^{q-\Omega n}}{(rn)^{q}} + \sum_{q=\Omega n}^{+\infty} \frac{(q)_p n^q q^a}{(rn)^{q}} \right) \nonumber \\
    &\leq \delta_{n, k} g(z)^k \xi^{n+o(n)} \left( \frac{(\Omega n + h)^{a+h+1}}{r^{\Omega n}} + \frac{(\Omega n+a+h)^{a+h}}{r^{\Omega n}(1-\frac{1}{r})^{a+h+1}} \right) \nonumber \\
    &\stackunder{$\leq$}{$\scriptscriptstyle{n\to+\infty}$} \Bigg( \Big(e^3 (2a+1) g(z)\Big)^\kappa \frac{\xi}{r^\Omega}\Bigg)^{n+o(n)}.
\end{align}


\smallskip
Recall that $\widetilde{\Lambda}_{n, (p, k)}(z_1)$ and $\Lambda_{n, (p, k)}(z_1)$ are defined by equations \eqref{définitiondeLambdapk} and \eqref{definitiondelambdafinal}, so that
\begin{align*}
    \Lambda_{n, (p, k)}(z_1) &= \frac{\delta_{n, k}}{(k-1)!}\den(z_1)^{(r+1)n}\widetilde{\Lambda}_{n, (p, k)}(z_1) \\
    &= \den(z_1)^{(r+1)n} \sum_{\ell=0}^{N-1} \hat{{f}}(\ell) (\mu^{-\ell} z_1)^{k-1} \frac{\delta_{n, k}}{(k-1)!}S_{n, p}^{[\infty](k-1)}(\mu^{-\ell} z_1).
\end{align*}
Since the upper bound \eqref{majorationdedeltafoisSsurkfact} holds for every $z$ such that $1/z \in\mathcal{D}$, we get
\begin{align*}
    \left| \Lambda_{n, (p, k)}(z_1) \right| &\stackunder{$\leq$}{$\scriptscriptstyle{n\to+\infty}$}  \den(z_1)^{(r+1)n} N  \max_{0\leq\ell\leq N-1}|\hat{{f}}(\ell)| \max(1, |z_1|)^{\kappa n} \Bigg( \Big(e^3(2a+1)g(z_1)\Big)^\kappa \frac{\xi}{r^\Omega} \Bigg)^{n+o(n)} \nonumber \\
                                                                        &\stackunder{$\leq$}{$\scriptscriptstyle{n\to+\infty}$} \Bigg( \den(z_1)^{r+1}\Big(e^3(2a+1) \max(1, |z_1|)g(z_1)\Big)^\kappa \frac{\xi}{r^\Omega}  \Bigg)^{n+o(n)}. 
\end{align*}
\end{proof}

\section{Application of a zero estimate}
\label{partie6}

In this section, we prove that the linear combinations $\Lambda_{n, (p, k)}(z_1)$ constructed in \autoref{paragraphenotationspourlasuite} satisfy assumption $(iii)$ of the linear independence criterion (\sref{Proposition}{critèred'indépendancelinéaire}). 

\smallskip
In \autoref{paragrapheénoncédeShidlovskii}, we state a version of ``Shidlovskii's lemma" in a general context.

In \autoref{paragraphestructureshidlovskii}, we state \sref{Proposition}{propositiondéduitedeshidlovskii}, which asserts that assumption $(iii)$ is satisfied by the linear combinations $\Lambda_{n, (p, k)}(z_1)$ under an assumption on the parameters $\omega, r, h$ and $\kappa$. The rest of \autoref{partie6} is dedicated to the proof of this proposition, whose strategy is detailed at the end of \autoref{paragraphestructureshidlovskii}.

\subsection{Statement of a ``Shidlovskii's lemma"}

\label{paragrapheénoncédeShidlovskii}

The generic term ``Shidlovskii's lemma" stands for a result based on Shidlovskii's ideas \cite{shidlovskii89}, that was built upon by different authors, notably by Bertrand and Beukers \cite{bertrandbeukers85}, Bertrand \cite{bertrand2012} and Fischler \cite{fischler2018}.

Let $d, \Delta\geq 1$ be integers. Consider a matrix $A \in M_d\big(\C(z)\big)$ and weighting polynomials $S_1, ..., S_{d}\in\C[z]$ of degree at most $\Delta$. To each solution $Y = \prescript{t}{}{(}y_1, ..., y_{d})$ of the differential system $Y' = AY$, we associate the remainder
\begin{equation*}
    \rho(Y)(z): = \sum_{i = 1}^d S_{i}(z)y_i(z).
\end{equation*}

According to \cite[Theorem 6.6, p.~181]{coddingtoncarlson97}, at any regular singular point $\sigma\in\C\cup\{\infty\}$ of $A$, the differential system $Y' = AY$ has a basis of local solutions belonging to the Nilsson class. They can thus be written locally around $\sigma$ as
\begin{equation}
    f(z) = \sum_{e\in\C}\sum_{s\in\Z_{\geq 0}} f_{e, s}(z)(z-\sigma)^e \log(z-\sigma)^s
\end{equation}
where $f_{e, s}$ is holomorphic and $f_{e, s}(\sigma) \neq 0$. If $\sigma = \infty$, $``(z - \infty)"$ is to be understood as ${1}/{z}$.

If $f$ is non-zero, we denote by $\mathrm{ord}_\sigma(f)$ the smallest element of the finite set $\big\{\mathfrak{Re}(e) \;\;\big|\;\; \exists s\in\Z_{\geq 0} \;\;\; f_{e, s} \neq 0 \big\}$. This is the order of vanishing of $f$ at $\sigma$. Note that logarithmic factors don't impact the order of $f$ at $\sigma$. In this context, we have the following inequality (see \cite[Corollary 6.1.2]{chapitre1}, \cite[Theorem 3.1]{fischler2018} for a proof):

\begin{theorem}
\label{Shidlovskii}
    (``Shidlovskii's lemma").
    Let $\Sigma$ be a finite set of points of $\C\cup\{\infty\}$. For every $\sigma\in\Sigma$, let $(Y_j)_{j\in J_\sigma}$ be a finite family of solutions of $Y' = AY$ in the Nilsson class at $\sigma$ and such that the remainders $\big(\rho(Y_j)\big)_{j\in J_\sigma}$ are $\C$-linearly independent functions. 

    There exists a constant $c_1$ which only depends on $A$ and $\Sigma$, such that 
    \begin{equation*}
        \sum_{\sigma\in\Sigma}\sum_{j\in J_\sigma} \mathrm{ord}_\sigma \rho(Y_j) \leq (\Delta+1)(d - \#J_\infty) + c_1,
    \end{equation*}
    where $d$ is the size of the square matrix $A \in M_d\big(\C(z)\big)$.
\end{theorem}

\subsection{A crucial proposition}
\label{paragraphestructureshidlovskii}

We fix $n \in \mathcal{N}$. The quantities $\Lambda_{n, (p, k)}(z_1)$, $\;(p, k)\in\llbracket 0, h\rrbracket \times \llbracket (r+1)n+1, \kappa n\rrbracket$, defined in \sref{Subsection}{paragraphenotationspourlasuite} are linear combinations of the numbers $\zeta_0^{<0>}, ..., \zeta_0^{<N-1>}, \zeta_1, ..., \zeta_{a+h}$. The rest of section $6$ is dedicated to the proof of \sref{Proposition}{propositiondéduitedeshidlovskii} below, that can be interpreted in the following way. By thinking of the $\Lambda_{n, (p, k)}(z_1)$ as linear forms evaluated at the point $\bm{\zeta} = \prescript{t}{}{(}\zeta_0^{<0>}, ..., \zeta_0^{<N-1>}, \zeta_1, ..., \zeta_{a + h})$, we observe that any point $\bm{x} = \prescript{t}{}{(}x_0^{<0>}, ..., x_0^{<N-1>}, x_1, ..., x_{a+h})$ annihilating all of these linear forms cannot be equal to $\bm{\zeta}$. Indeed, such an $\bm{x}$ has to annihilate another family of linear forms, among which at least one is not annihilated by $\bm{\zeta}$. This property is required to apply the linear independence criterion (\sref{Proposition}{critèred'indépendancelinéaire}) in \sref{Section}{partie7}.

\begin{proposition}
\label{propositiondéduitedeshidlovskii}
    Assume that $(h+1)(\kappa-r-1)N > a$ and that $n$ is large enough. Then for every $\bm{x} = \prescript{t}{}{(}x_0^{<0>}, ..., x_0^{<N-1>}, x_1, ..., x_{a+h}) \in \C^{a+h+N}$ such that
    \begin{equation}
    \label{THEequation}
        \forall (p, k)\in\llbracket 0, h\rrbracket \times \llbracket (r+1)n+1, \kappa n\rrbracket \;\;\; \sum_{m=0}^{N-1} \lambda_{n, 0, (p, k)}^{<m>}(z_1) x_0^{<m>} + \sum_{i=1}^{a + h} \lambda_{n, i, (p, k)}(z_1) x_i = 0,
    \end{equation}
    we have
    \begin{equation*} \forall m\in\llbracket 0, N-1\rrbracket,\hspace{0.5cm}
        x_0^{<m>} = 0.
    \end{equation*}
\end{proposition}

To prove this proposition, we adopt a similar strategy as \cite[section 4.5]{fischler2026}, even though the multiplicity estimate will be applied to more points. We first define the integer
\begin{equation}
\label{définitiondebn}
    b_n: = \max \Big\{ i \in \llbracket 1, a \rrbracket \;\;\Big|\;\; P_{n, i} \neq 0 \Big\},
\end{equation}
which exists because the polynomials $P_{n, i}$ are defined by \eqref{definitiondesPki} and the integers $c_{n, i, j}$ obtained from \sref{Proposition}{existencedescij} are not all zero. This integer satisfies $P_{n, b_n} \neq 0$ and $P_{n, i} = 0$ for $i > b_n$. Because of definitions \eqref{définitiondesQip}, \eqref{récurrencevérifiéeparlesQipk} and \eqref{definitiondespetitslambdas}, the polynomials $Q_{n, i, (p, k)}$ and the coefficients $\lambda_{n, i, (p, k)}$ all vanish for $b_n+p+1\leq i\leq a+h$. In particular, they vanish for $i > b_n + h$.

\smallskip
We consider, for $\;0\leq\ell\leq N-1\;$, the polynomials 
\begin{equation}
    \label{definitiondesQ0l}
        Q_{n, 0, (p, k)}^{\{\ell\}}(z) := \mu^{\ell(k-1)}Q_{n, 0, (p, k)}(\mu^\ell z).
\end{equation}

Equation $(iii.b)$ of \sref{Lemma}{formulesd'inversiondeFourier} evaluated at $z=z_1$ yields
\begin{equation*}
\label{fourieren-1}
    z_1^m Q_{n, 0, (p, k)}^{<m>}(z_1) = \frac{(z_1)^{k-1}}{N}\sum_{\ell=0}^{N-1}\mu^{-\ell m} Q_{n, 0, (p, k)}^{\{\ell\}}(z_1), \hspace{1cm}0\leq m\leq N-1.
\end{equation*}
Thus, equation \eqref{definitiondespetitslambdas} yields
\begin{equation}
\label{expressiondespetitslambdasaveclesQ}
    \begin{split}
        &\lambda_{n, i, (p, k)}(z_1) = \frac{\delta_{n, k}}{(k-1)!} \den(z_1)^{(r+1)n} z_1^{k-1} Q_{n, i, (p, k)}(z_1), \hspace{2.6cm}1\leq i\leq {b_n} +h,\\
        &\lambda_{n, 0, (p, k)}^{<m>}(z_1) = \frac{\delta_{n, k}}{(k-1)!} \den(z_1)^{(r+1)n}\frac{z_1^{k-1}}{N} \sum_{\ell = 0}^{N-1} \mu^{-\ell m} Q_{n, 0, (p, k)}^{\{\ell\}}(z_1), \hspace{1cm} 0\leq m\leq N-1.
    \end{split}
\end{equation}

\smallskip
Assume by contradiction that \sref{Proposition}{propositiondéduitedeshidlovskii} is false. Recall that by definition of $b_n$, all of the terms in the second sum with $i > b_n + h$ in \eqref{THEequation} vanish. Hence there exist $\bm{x}_n = \prescript{t}{}{(}x_{n, 0}^{<0>}, ..., x_{n, 0}^{<N-1>}, x_{n, 1}, ..., x_{n, b_n+h})\in\C^{b_n+h+N}$ with
\begin{equation}
\label{hypothèsesurx}
    \forall (p, k)\in\llbracket 0, h\rrbracket \times \llbracket (r+1)n+1, \kappa n\rrbracket \;\;\; \sum_{m=0}^{N-1} \lambda_{n, 0, (p, k)}^{<m>} x_0^{<m>} + \sum_{i=1}^{b_n + h} \lambda_{n, i, (p, k)} x_i = 0
\end{equation}
and the coefficients $x_0^{<m>}, 0\leq m\leq N-1$, not all zero. Then, since $(\mu^{-\ell m})_{0\leq \ell, m\leq N-1}$ is an invertible Vandermonde matrix, there exists $\ell_0\in\llbracket 0, N-1\rrbracket$ such that
\begin{equation}
\label{hypothèsesurl0}
    \sum_{m=0}^{N-1} \mu^{-\ell_0 m} x_0^{<m>} \neq 0.
\end{equation}

We define for $0\leq\ell\leq N-1$ and $1\leq i\leq b_n+h$:
\begin{equation}
\label{definitiondesxi}
    \xi_{n, 0}^{\{\ell\}} := \sum_{m=0}^{N-1} \frac{\mu^{-\ell m}}{N} x_{n, 0}^{<m>}, \hspace{1cm} \xi_{n, i} := x_{n, i}.
\end{equation}
By using \eqref{expressiondespetitslambdasaveclesQ} and \eqref{hypothèsesurx}, we readily obtain
\begin{equation}
\label{hypothesesurlesxini}
    \forall (p, k) \in \llbracket 0, h \rrbracket \times \llbracket (r+1)n+1, \kappa n \rrbracket \hspace{1cm}
    \sum_{\ell=0}^{N-1} \xi_{n, 0}^{\{\ell\}}Q_{n, 0, (p, k)}^{\{\ell\}}(z_1)  + \sum_{i=1}^{{b_n} +h}Q_{n, i, (p, k)}(z_1)\xi_{n, i} = 0. 
\end{equation}

\smallskip
In \sref{Subsection}{constructiondesfp}, we use the coefficients $\xi_{n, 0}^{\{\ell\}}$ and $\xi_{n, i}$ to construct functions $f_{n, p}$ with a high order of vanishing at $z = z_1$. They can be shifted to obtain functions $f_{n, p}(\mu^\ell z)$ with a high order of vanishing at $z = \mu^{-\ell} z_1$, $\;\;0\leq \ell\leq N-1$.

From these functions, we construct in \sref{Subsection}{constructiondesrhoq} new functions $\rho_{n, q}(\mu^\ell z)$ with the same order of vanishing, and that can be seen as remainders in the context of \sref{Subsection}{paragrapheénoncédeShidlovskii}.

We construct in \sref{Subsection}{constructiondestaubeta} other functions $\tau_{n, u}$ that can be seen as remainders for the same differential system, and with a high order of vanishing at $0, 1$ and $\infty$.

All of these remainders with high orders of vanishing allow us to obtain, for $n$ large enough, a contradiction in \sref{Subsection}{appplicationdeShidlovskii} when applying ``Shidlovskii's lemma" (\sref{Theorem}{Shidlovskii}). This completes the proof of \sref{Proposition}{propositiondéduitedeshidlovskii}.

\subsection{Construction of the functions $f_{n, p}$}
\label{constructiondesfp}

We consider an open and simply connected neighborhood $\mathscr{D}$ of $z_1$ such that $\mathscr{D}^{-1} \subseteq \mathcal{D}$. For $\log(\mu^\ell z)$ to be well defined for any $z\in\mathscr{D}$ and $\ell\in\llbracket 0, N-1\rrbracket$, we also assume that $\mu^\ell\mathscr{D} \subset \mathcal{U}$ for $\ell\in\llbracket 0, N-1\rrbracket$, where the open set $\mathcal{U}$ is defined in \sref{Notation}{notations}. We define by induction holomorphic functions $g_{n, i}: \mathscr{D} \to \C$ as follows:
\begin{equation*}
    \begin{cases}
        g_{n, 1}(z_1) = \xi_{n, 1}, &g'_{n, 1}(z) = \sum_{\ell=0}^{N-1} \frac{\xi_{n, 0}^{\{\ell\}}}{z(1-\mu^\ell z)}, \\
        g_{n, i}(z_1) = \xi_{n, i}, &g'_{n, i}(z) = \frac{-1}{z}g_{n, i-1}(z), \hspace{2cm} 2\leq i\leq b_n+h.
    \end{cases}
\end{equation*}

For $0 \leq p \leq h$, we then define
\begin{equation}
\label{définitiondesfp}
    f_{n, p}(z) := T_{n, p}(z) + \sum_{i=1}^{{b_n} +h} Q_{n, i, p}(z)g_{n, i}(z),
\end{equation}
where $T_{n, p}(z)$ is the opposite of the Taylor polynomial of order $(r+1)n-1$ of $\sum_{i=1}^{{b_n} +h} Q_{n, i, p}(z)g_{n, i}(z)$ at $z_1$, so that $\mathrm{ord}_{z_1}(f_{n, p}) \geq (r+1)n$.

It turns out that the order of vanishing of $f_{n, p}$ at $z_1$ is much larger than this. Indeed, the substitution $z \xleftarrow[]{} \mu^\ell z$ in \eqref{récurrencevérifiéeparlesQipk}, along with equations \eqref{definitiondesQ0l} and \eqref{lesQnipksontdansQzN}, yields
\begin{equation*}
        Q_{n, 0, (p, k+1)}^{\{\ell\}}(z) = (Q_{n, 0, (p, k)}^{\{\ell\}})'(z) + \frac{1}{z(1-\mu^{\ell} z)}Q_{n, 1, (p, k)}, \hspace{1cm}Q_{n, 0, (p, 1)}^{\{\ell\}} = 0.
\end{equation*}
Thus, by \eqref{définitiondesfp} and the definition of the functions $g_{n, i}$, one can show by induction on $k \geq 1$:
\begin{equation*}
    f_{n, p}^{(k-1)}(z) = T_{n, p}^{(k-1)}(z) + \sum_{\ell=0}^{N-1} \xi_{n, 0}^{\{\ell\}}Q_{n, 0, (p, k)}^{\{\ell\}}(z)  + \sum_{i=1}^{{b_n} +h}Q_{n, i, (p, k)}(z)g_{n, i}(z).
\end{equation*}
For $(r+1)n+1 \leq k \leq \kappa n$, The polynomial $T_{n, p}^{(k-1)}$ vanishes because $\mathrm{deg}\big(T_{n, p}\big) \leq (r+1)n-1$. It follows that $f_{n, p}^{(k-1)}(z_1) = 0$ because of \eqref{hypothesesurlesxini}, since $g_{n, i}(z_1) = \xi_{n, i}$. Thus $\mathrm{ord}_{z_1}(f_{n, p}) \geq \kappa n, \;\; 0\leq p\leq h$.

the functions $f_{n, p}$ can be seen in the context of \autoref{paragrapheénoncédeShidlovskii} as the remainders associated with the solution $\prescript{t}{}{(}1, g_{n, 1}, ..., g_{n, {b_n} +h})$ of the differential system
\begin{equation}
\label{premiersystemediff}
    Y' = \begin{bmatrix}
        0 &0 &0 &\hdots &0 &0 \\
        \sum_{\ell=0}^{N-1} \frac{\xi_{n, 0}^{\{\ell\}}}{z(1-\mu^\ell z)} &0 &0 &\hdots &0 &0 \\
        0 &\frac{-1}{z} &0 &\hdots &0 &0 \\
        0 &0 &\frac{-1}{z} &\hdots &0 &0 \\
        \vdots &\vdots &\vdots &\ddots &\vdots &\vdots \\
        0 &0 &0 &\hdots &\frac{-1}{z} &0 
    \end{bmatrix}Y
\end{equation}
weighted by the polynomials $T_{n, p}, Q_{n, 1, p}, ..., Q_{n, {b_n} +h, p}$. This is not suitable to apply \sref{Theorem}{Shidlovskii}: we have only one solution $(g_{n, i})_i$ weighted by different families $(Q_{n, i, p})_i$ depending on $p$ to obtain the remainders $f_{n, p}$. We should instead have different solutions $(y_{n, i, q})_i$, weighted by only one family $(S_{n, i})_i$ independent of $q$ to obtain some remainders $\rho_{n, q}$.

\smallskip
We will construct all of these objects related to the differential system \eqref{grossystemediffcompliqué} in \autoref{constructiondesrhoq}. Before that, notice that we can easily obtain $N$ times as many functions with high orders of vanishing by considering the substitutions $z \xleftarrow[]{} \mu^\ell z$. Indeed, we have
\begin{equation}
\label{annulationdesfp}
    \forall p\in\llbracket 0, h \rrbracket, \;\forall \ell \in \llbracket 0, N-1 \rrbracket, \hspace{0.8cm} \mathrm{ord}_{\mu^{-\ell} z_1}\Big(f_{n, p}(\mu^\ell z)\Big) \geq \kappa n.
\end{equation}
The functions $g_{n, i}(\mu^\ell z)$ obey the same differentiation rules as $g_{n, i}(z)$, except for $i = 1$. Notice that by adding $N-1$ rows of zeroes at the top of the matrix in \eqref{premiersystemediff} and then replacing its first column by the $N$ columns 
\begin{equation*}
    \begin{bmatrix}
        0 \\
        \vdots \\
        0 \\
        \sum_{m=0}^{N-1} \frac{\xi_{n, 0}^{\{m\}}}{z(1-\mu^{m+\ell} z)} \\
        0 \\
        0 \\ 
        \vdots \\
        0
    \end{bmatrix}, \hspace{2cm} 0\leq \ell\leq N-1,
\end{equation*}
\noindent we obtain a differential system that admits $\prescript{t}{}{\big(}0, ..., 0, 1, 0, ..., 0, g_{n, 1}(\mu^\ell z), ..., g_{n, {b_n} +h}(\mu^\ell z)\big)$, as solutions, $0\leq\ell\leq N-1$, with the $1$ on the $(\ell+1)$-th slot.
The functions $f_{n, p}(\mu^\ell z)$ can be seen as the remainders associated with these solutions, weighted by the polynomials $T_{n, p}(z), T_{n, p}(\mu z), ..., \break T_{n, p}(\mu^{N-1} z), Q_{n, 1, p}, ..., Q_{n, {b_n} +h, p}$.

\subsection{Construction of the remainders $\rho_{n, q}$}
\label{constructiondesrhoq}

We define the functions
\begin{equation}
\label{definitiondesrhoq}
    \rho_{n, q}(z) := \sum_{p=0}^q \binom{q}{p} \Big(-\log(z)\Big)^{q-p}f_{n, p}(z), \hspace{2cm} 0\leq q\leq h.
\end{equation}
This subsection introduces a new differential system, vectors of solutions and polynomials such that $\rho_{n, q}(\mu^\ell z)$ can be seen as remainders.

\smallskip
Let $0_{u, v}$, resp. $0_u$, denote the zero matrix of $M_{u, v}(\C)$, resp. $M_u(\C)$. We denote by $J_u$ the matrix of $M_u(\C)$ with $1$ on the sub-diagonal and $0$ everywhere else. We consider the differential system $Y' = \mathscr{A}_n Y$ with
\begin{equation}
    \label{grossystemediffcompliqué}
    \mathscr{A}_n :=
        \begin{bmatrix}
            0_N &0_N &0_N &\hdots &0_N &0_N &0_{N, {b_n} } \\
            \frac{-1}{z}I_N &0_N &0_N &\hdots &0_N &0_N &0_{N, {b_n} } \\
            0_N &\frac{-1}{z}I_N &0_N &\hdots &0_N &0_N &0_{N, {b_n} } \\
            0_N &0_N &\frac{-1}{z}I_N &\hdots &0_N &0_N &0_{N, {b_n} } \\
            \vdots &\vdots &\vdots &\ddots &\vdots &\vdots &\vdots \\
            0_N &0_N &0_N &\hdots &\frac{-1}{z}I_N &0_N &0_{N, {b_n} } \\
            0_{{b_n} , N} &0_{{b_n} , N} &0_{{b_n} , N} &\hdots &0_{{b_n} , N} &\mathscr{B}_n &\frac{-1}{z}J_{{b_n} } 
        \end{bmatrix} \in M_{N(h+1)+{b_n} }\big(\C(z)\big).
    \end{equation}
Notice that the matrix $\mathscr{A}_n$ only depends of $n$ through the integer $b_n$ and the block $\mathscr{B}_n\in M_{{b_n} , N}\big(\C(z)\big)$ whose columns are
\begin{equation*}
    \begin{bmatrix}
        \sum_{m=0}^{N-1} \frac{\xi_{n, 0}^{\{m\}} z}{z(1-\mu^{m+\ell} z)} \\
        0 \\
        \vdots \\
        0
    \end{bmatrix}, \hspace{2cm}{0\leq \ell\leq N-1}.
\end{equation*}

\smallskip
We also consider the polynomials
\begin{equation*}
    \begin{cases}
        S_{n, i}(z) = \frac{1}{(h-i)!} T_{n, h-i}(z), &0\leq i\leq h, \\
        S_{n, i}(z) = z^{rn}P_{n, i-h}(z), &h+1\leq i\leq {b_n} +h,
    \end{cases}
\end{equation*}
and for $0\leq q\leq h$ the functions 
\begin{equation*}
    \begin{cases}
        y_{n, i, q}(z) = 0, &\hspace{1cm}0\leq i\leq h-q-1, \\
        y_{n, i, q}(z) = \frac{q!}{(i+q-h)!}\Big(-\log(z)\Big)^{i+q-h}, &\hspace{1cm}h-q\leq i\leq h, \\
        y_{n, i, q}(z) = \sum_{p=0}^q \binom{q}{p} (-1)^p(i-h)_p\Big(-\log(z)\Big)^{q-p}g_{n, i-h+p}(z), &\hspace{1cm}h+1\leq i\leq {b_n} +h.
    \end{cases}
\end{equation*}
We organize them inside column vectors by defining
\begin{align*}
    \mathcal{S}_n := \prescript{t}{}{\Big(}S_{n,0}(z), S_{n,0}(\mu z), ..., S_{n,0}(\mu^{N-1}z), ...\;...,  S_{n,h}(z), &S_{n,h}(\mu z), ..., S_{n,h}(\mu^{N-1}z), \\
    &S_{n, h+1}(z), S_{n, h+2}(z), ..., S_{n, {b_n} +h}(z)\Big)
\end{align*}
and, for $0\leq q\leq h$ and $0\leq\ell\leq N-1$,
\begin{align*}
        Y_{n, q}^{\{\ell\}}(z) := \prescript{t}{}{\Big(}0, \;\;..., \;\;y_{n, 0, q}(\mu^\ell z), \;\;..., \;\;0, \;\;...\;\;..., \;\;0, \;\;..., \;\;&y_{n, h, q}(\mu^\ell z), \;\;..., \;\;0, \\
        &y_{n, h+1, q}(\mu^\ell z), ..., y_{n, {b_n} +h, q}(\mu^\ell z)\Big).
\end{align*}
In the definition of $Y_{n, q}^{\{\ell\}}(z)$, up to the index $i=h$, each chunk has length $N$ and contains $N-1$ zeroes and $y_{n, i, q}(\mu^\ell z)$ at the $(\ell+1)$-th slot. Starting from the index $i=h+1$, each chunk has length $1$ and only contains $y_{n, i, q}(\mu^\ell z)$.

Notice that because of the definition \eqref{definitiondesPki} of the polynomials $P_{n, i}$, the polynomials $S_{n, i}$ with $i \geq h+1$ satisfy $S_{n, i}(\mu^\ell z) = S_{n, i}(z)$ for every $\ell\in\llbracket 0, N-1\rrbracket$.

\begin{lemma}
    In this context, 
    \begin{enumerate}
        \item[$(i)$] for $0\leq q\leq h$ and $0\leq\ell\leq N-1$, $Y_{n, q}^{\{\ell\}}$ is a solution of the differential system $Y' = \mathscr{A}_n Y$;
        \item[$(ii)$] the functions $\rho_{n, q}(\mu^\ell z)$ are the remainders associated with these solutions and weighted by the polynomials $\mathcal{S}_n$:
        \begin{equation*}
            \forall q\in\llbracket 0, h\rrbracket \;\;\forall\ell\in\llbracket 0, N-1\rrbracket \;\;\;\;\;\;\rho_{n, q}(\mu^\ell z) = \sum_{i=0}^{{b_n} +h} S_{n, i}(\mu^\ell z) y_{n, i, q}(\mu^\ell z) = \prescript{t}{}{\mathcal{S}_n}\cdot Y_{n, q}^{\{\ell\}} \:;
        \end{equation*}
        \item[$(iii)$] for $\ell\in\llbracket 0, N-1 \rrbracket$, the family $\Big(\rho_{n, q}(\mu^\ell z)\Big)_{0\leq q\leq h}$ is linearly independent over $\C$.
    \end{enumerate}
\end{lemma}

\begin{proof}
    $(i)\;$We fix $q\in\llbracket 0, h\rrbracket$ and $\ell\in\llbracket 0, N-1\rrbracket.$
    \begin{enumerate}
        \item[$\bullet$] If $i=0$, $\;y_{n, 0, q}(\mu^\ell z)$ is constant, and its derivative is zero.
        \item[$\bullet$] If $1\leq i\leq h-q$, $\;y_{n, i, q}(\mu^\ell z)$ is constant and its derivative is zero, and $\frac{-1}{z}y_{n, i-1, q}(\mu^\ell z)$ is zero as well.
        \item[$\bullet$] If $h-q+1\leq i\leq h$, we have
            \begin{align*}
                \frac{d}{dz}\Big[y_{n, i, q}(\mu^\ell z)\Big] &= \frac{q!}{(i+q-h)!}(i+q-h)\frac{-\mu^\ell}{\mu^\ell z}\Big(-\log(\mu^\ell z)\Big)^{i+q-h-1} \\
                                                        &= \frac{-1}{z}\frac{q!}{(i+q-h-1)!}\Big(-\log(\mu^\ell z)\Big)^{i+q-h-1} \\
                                                        &= \frac{-1}{z}y_{n, i-1, q}(\mu^\ell z).
            \end{align*}
        \item[$\bullet$] If $i=h+1$, we have
            \begin{align*}
                \frac{d}{dz}\Big[y_{n, h+1, q}(\mu^\ell z)\Big] &= \Big(-\log(\mu^\ell z)\Big)^q \Bigg(\sum_{m=0}^{N-1} \frac{\xi_{n, 0}^{\{m\}}}{z(1-\mu^{m+\ell} z)} \Bigg) \\
                &\hspace{5mm}+ \sum_{p=1}^q \binom{q}{p} (-1)^p p!\Big(-\log(\mu^\ell z)\Big)^{q-p} \cdot\frac{-1}{z}g_{n, p}(\mu^\ell z) \\
                &\hspace{5mm}+ \sum_{p=0}^{q-1} \binom{q}{p}(-1)^p p! (q-p) \frac{-1}{z}\Big(-\log(\mu^\ell z)\Big)^{q-p-1}g_{n, p+1}(\mu^\ell z) \\
                &= \Bigg(\sum_{m=0}^{N-1} \frac{\xi_{n, 0}^{\{m\}}}{z(1-\mu^{m+\ell} z)} \Bigg) y_{n, h, q}(\mu^\ell z).
            \end{align*}
            Indeed, the two sums over $p$ are opposite to each other, as can be seen by applying the substitution $p\xleftarrow[]{} p+1$ in the second sum.
            
        \item[$\bullet$] Lastly, if $h+2\leq i\leq b_n+h$, we have
            \begin{align*}
                &\frac{d}{dz}\Big[y_{n, i, q}(\mu^\ell z)\Big] \\
                = \hspace{3mm}&\sum_{p=0}^q \binom{q}{p}(-1)^p(i-h)_p\Big(-\log(\mu^\ell z)\Big)^{q-p}\cdot\frac{-1}{z}g_{n, i-h+p-1}(\mu^\ell z) \\
                &+\hspace{1mm}\sum_{p=0}^{q-1} \binom{q}{p} (-1)^p (i-h)_p (q-p)\frac{-1}{z}\Big(-\log(\mu^\ell z)\Big)^{q-p-1}g_{n, i-h+p}(\mu^\ell z) \\
                = \hspace{3mm}&\frac{-1}{z}\Big(-\log(\mu^\ell z)\Big)^qg_{n, i-h-1}(\mu^\ell z) \\
                &+\hspace{1mm}\frac{-1}{z}\sum_{p=1}^q \frac{q!}{p!(q-p)!}(-1)^p(i-h)_p\Big(-\log(\mu^\ell z)\Big)^{q-p}g_{n, i-h+p-1}(\mu^\ell z) \\
                &+\hspace{1mm}\frac{-1}{z}\sum_{p=1}^q \frac{q!}{(p-1)!(q-p)!}(-1)^{p-1}(i-h)_{p-1}\Big(-\log(\mu^\ell z)\Big)^{q-p} g_{n, i-h+p-1}(\mu^\ell z) \\ 
                = \hspace{3mm}&\frac{-1}{z}\Big(-\log(\mu^\ell z)\Big)^qg_{n, i-h-1}(\mu^\ell z) \\  
                &+\hspace{1mm}\frac{-1}{z}\sum_{p=1}^q \frac{q!}{(q-p)!}(-1)^p\bigg(\frac{(i-h)_p}{p!}-\frac{(i-h)_{p-1}}{(p-1)!}\bigg)\Big(-\log(\mu^\ell z)\Big)^{q-p}g_{n, i-h+p-1}(\mu^\ell z) \\
                = \hspace{3mm}&\frac{-1}{z}\sum_{p=0}^q \binom{q}{p}(-1)^p(i-h-1)_p\Big(-\log(\mu^\ell z)\Big)^{q-p}g_{n, i-h+p-1}(\mu^\ell z) \\
                = \hspace{3mm}&\frac{-1}{z}y_{n, i-1, q}(\mu^\ell z),
            \end{align*}
            by using the formula $\frac{(i-h)_p}{p!}-\frac{(i-h)_{p-1}}{(p-1)!} = \frac{(i-h-1)_p}{p!}$.
    \end{enumerate}

    \smallskip
    $(ii)\;$ We fix $\ell\in\llbracket 0, N-1\rrbracket$. Recall that for $i\geq 1$, the polynomials $Q_{n, i, p}$ defined by \eqref{définitiondesQip} vanish when $i\notin\llbracket p+1, {b_n} +p\rrbracket$ and are in $\Q[z^N]$ otherwise, so that $Q_{n, i, p}(\mu^\ell z) = Q_{n, i, p}(z)$. We use equations \eqref{définitiondesfp} and \eqref{definitiondesrhoq} to compute:
    \begin{small}
    \begin{align*}
        \rho_{n, q}(\mu^\ell z) &= \sum_{p=0}^q \binom{q}{p} \big(-\log(\mu^\ell z)\big)^{q-p} \Bigg[T_{n, p}(\mu^\ell z) + \sum_{i=p+1}^{{b_n} +p} Q_{n, i, p}(z)g_{n, i}(\mu^\ell z) \Bigg] \\
                            &= \sum_{p=0}^q \frac{1}{p!}T_{n, p}(\mu^\ell z) \frac{q!}{(q-p)!}\big(-\log(\mu^\ell z)\big)^{q-p} \\
                            &\hspace{0.5cm}+ \sum_{p=0}^q \binom{q}{p} \big(-\log(\mu^\ell z)\big)^{q-p} \sum_{i=h+1}^{{b_n} +h} Q_{n, i-h+p, p}(z)g_{n, i-h+p}(\mu^\ell z) \hspace{20mm}(i \xleftarrow[]{} i+h-p) \\
                            &= \sum_{i=h-q}^{h} \frac{1}{(h-i)!}T_{n, h-i}(\mu^\ell z) \frac{q!}{(i+q-h)!} \big(-\log(\mu^\ell z)\big)^{i+q-h} \hspace{35mm}(i = h - p) \\
                            &\hspace{0.5cm}+ \sum_{i = h+1}^{{b_n} +h} z^{rn}P_{n, i-h}(z)\sum_{p=0}^q \binom{q}{p} (-1)^p(i-h)_p \big(-\log(\mu^\ell z)\big)^{q-p}g_{n, i-h+p}(\mu^\ell z) \hspace{4mm}\text{\big(from \eqref{définitiondesQip}\big)} \\
                            &= \sum_{i=0}^{{b_n} +h} S_{n, i}(\mu^\ell z) y_{n, i, q}(\mu^\ell z) = \prescript{t}{}{\mathcal{S}_n}\cdot Y_{n, q}^{\{\ell\}}.
    \end{align*}
    \end{small}

    \smallskip
    $(iii)\;$ We fix $\ell\in\llbracket 0, N-1 \rrbracket$ and assume by contradiction that there exist $\nu_0, ..., \nu_h \in\C$ not all zero such that $\sum_{q=0}^h \nu_q\rho_{n, q}(\mu^\ell z) = 0$. We use the result of $(ii)$ above and swap the sums to obtain
    \begin{equation*}
        \sum_{i=0}^{{b_n} +h} S_{n, i}(\mu^\ell z)\sum_{q=0}^h \nu_q y_{n, i, q}(\mu^\ell z) = 0.
    \end{equation*}
    Consider the functions $y_{n, i}(z) := \sum_{q=0}^h \nu_q y_{n, i, q}(z)$. On the one hand, they satisfy
    \begin{equation}
    \label{initialisationdelarécurrence}
        \sum_{i=0}^{{b_n} +h} S_{n, i}(\mu^\ell z)y_{n, i}(\mu^\ell z) = 0.
    \end{equation}
    On the other hand, they are the coordinates of the vector $Y_{n}^{\{\ell\}}(z)$ defined by
    \begin{align*}
        Y_{n}^{\{\ell\}}(z):&= \sum_{q=0}^h \nu_q Y_{n, q}^{\{\ell\}}(z) \\
        &= \prescript{t}{}{\Big(}0, ..., y_{n, 0}(\mu^\ell z), ..., 0, ...\;..., 0, ..., y_{n, h}(\mu^\ell z), ..., 0, \\
        &\hspace{8.5cm}y_{n, h+1}(\mu^\ell z), ..., y_{n, {b_n} +h}(\mu^\ell z)\Big),
    \end{align*}
    which is still a solution of the linear differential system $Y' = \mathscr{A}_n Y$.

    \smallskip
   We denote by $q_{\max}$ the highest index $q$ with $\nu_q \neq 0$, so that 
   \begin{align*}
       y_{n, h-q_{\max}}(z) &= \sum_{q=0}^{q_{\max}} \nu_q y_{n, h-q_{\max}, q}(z) \\
                            &= \nu_{q_{\max}} y_{n, h-q_{\max}, q_{\max}}(z) \hspace{2.5cm}\text{(since $y_{n, i, q} = 0$ when $i<h-q$)} \\
                            &= \nu_{q_{\max}} q_{\max} ! \neq 0.
   \end{align*}
   Let $i_{\min} = h-q_{\max}$, so that $y_{n, i_{\min}}(z)$ is a non-zero constant and $y_{n, i} = 0$ for $i<i_{\min}$.
    We prove by decreasing induction on $i_{\max}\in\llbracket i_{\min}, {b_n} +h\rrbracket$ the following fact:
    \begin{equation*}
        H_{i_{\max}}: \hspace{5mm}\text{\guillemotleft \hspace{2mm}There exist polynomials $R_{i_{\max}, i} \in \C[z], \;i_{\min}\leq i\leq i_{\max},\;$ such that $R_{i_{\max}, i_{\max}} \neq 0$} 
    \end{equation*}
    \begin{equation*}
        \text{and $\sum_{i=i_{\min}}^{i_{\max}} R_{i_{\max}, i}(z) y_{n, i}(\mu^\ell z) = 0$}\hspace{8mm} (*)\hspace{2mm}\text{\guillemotright}.
    \end{equation*}
    According to \eqref{initialisationdelarécurrence}, this holds for $i_{\max} = {b_n} +h$ by letting $R_{{b_n} +h, i}(z) = S_{n, i}(\mu^\ell z)$. The polynomial $R_{b_n+h, b_n+h} = S_{n, b_n+h} = z^{rn}P_{n, b_n}$ is indeed non-zero, because of the definition \eqref{définitiondebn} of $b_n$.
    Now, assume that $H_{i_{\max}}$ holds for some index $i_{\max}\in\llbracket i_{\min} + 1, {b_n} +h\rrbracket$. Let $v$ be the degree of $R_{i_{\max}, i_{\max}}$. By differentiating $v+1$ times equation $(*)$, we get
    \begin{equation*}
        \sum_{i=i_{\min}}^{i_{\max}} \widetilde{R}_{i_{\max}-1, i}(z)y_{n, i}(\mu^\ell z) = 0,
    \end{equation*}
    where the expression of the rational functions $\widetilde{R}_{i_{\max}-1, i}$ can be computed by using the Leibniz rule and the fact that $Y_n^{\{\ell\}}$ is a solution of $Y'=\mathscr{A}_nY$, where $\mathscr{A}_n$ is given by \eqref{grossystemediffcompliqué}. In particular, we notice that
     
    \begin{enumerate}[left=0.5pt]
        \item[$\bullet$] for $i_{\min}\leq i\leq i_{\max}$, the only possible poles for $\widetilde{R}_{i_{\max}-1, i}$ are $0, 1, \mu, \mu^2, ..., \mu^{N-1}$, with multiplicity at most $v+1$ each.
        \item[$\bullet$] $\widetilde{R}_{i_{\max}-1, i_{\max}} = R_{i_{\max}, i_{\max}}^{(v+1)} = 0.$
        \item[$\bullet$] 
        $\widetilde{R}_{i_{\max}-1, i_{\max}-1} =$ 
        \begin{equation*}
            \begin{cases}
                R_{i_{\max}, i_{\max}-1}^{(v+1)} + \sum_{u=0}^v \left(\frac{d}{dz}\right)^{v-u} \left(\frac{-1}{z}R_{i_{\max}, i_{\max}}^{(u)}\right) &\text{if $i_{\max}\neq h+1$}, \\
                R_{i_{\max}, i_{\max}-1}^{(v+1)} + \sum_{u=0}^v \left(\frac{d}{dz}\right)^{v-u} \Bigg(\bigg(\sum_{m=0}^{N-1}\frac{\xi_{n, 0}^{\{m\}}}{z(1-\mu^{m+\ell}z)} \bigg)R_{i_{\max}, i_{\max}}^{(u)}\Bigg) &\text{if $i_{\max} = h+1$}.
            \end{cases}
        \end{equation*}
    \end{enumerate}

    This last rational function is non zero. To prove it, it is enough to find a point in the complex plane where the term with $u = v$ in the sum has a non-zero residue. Indeed, no other term will be able to cancel this residue out, because polynomials and derivatives of rational functions have zero residue.
    \begin{enumerate}
        \item[\textrightarrow] In the case $i_{\max} \neq h+1$, the term in the sum with $u = v$ has residue $-R_{i_{\max}, i_{\max}}^{(v)}$ at $z=0$. This is a non-zero constant because of $H_{i_{\max}}$, since $v = \mathrm{deg}(R_{i_{\max}, i_{\max}})$.
        \item[\textrightarrow] In the case $i_{\max} = h + 1$, the term in the sum with $u = v$ has residue $-\xi_{n, 0}^{ \{ \widetilde{\ell} \} }R_{i_{\max}, i_{\max}}^{(v)}$ at $z=\mu^{-(\ell+\widetilde{\ell})}$. Using definition \eqref{definitiondesxi}, we compute:
        \begin{equation*}
            \xi_{n, 0}^{\{\widetilde{\ell}\}} = \frac{1}{N}\sum_{m=0}^{N-1} \mu^{-\widetilde{\ell} m}x_0^{<m>}.
        \end{equation*}
        According to assumption \eqref{hypothèsesurl0}, this is non-zero for $\widetilde{\ell} = \ell_0$.
    \end{enumerate}
     We thus obtain suitable polynomials for $H_{i_{\max}-1}$ by letting, for $i_{\min}\leq i\leq i_{\max}-1$:
    \begin{equation*}
        R_{i_{\max}-1, i} := z^{v+1}\prod_{m=0}^{N-1}(1-\mu^m z)^{v+1} \widetilde{R}_{i_{\max}-1, i}.
    \end{equation*}
    This concludes the inductive proof.

    \smallskip
    Now, $H_{i_{\min}}$ provides a non-zero polynomial $R_{i_{\min}, i_{\min}}$ such that $R_{i_{\min}, i_{\min}}(z)y_{n, i_{\min}} = 0$, which is a contradiction because $y_{n, i_{\min}}$ is a non-zero constant, by definition of $i_{\min}$.
    That means there cannot exist coefficients $\nu_0, ..., \nu_h\in\C$ not all zero such that $\sum_{q=0}^h \nu_q\rho_{n, q}(\mu^\ell z) = 0$.
\end{proof}

\subsection{Construction of the remainders $\tau_{n, u}$}
\label{constructiondestaubeta}

We consider for $1\leq u\leq {b_n} $ the solutions $\Psi_{n, u} = \prescript{t}{}{(}0, ..., 0, \psi_{n, 1, u}, ..., \psi_{n, {b_n} , u})$ of the differential system $Y' = \mathscr{A}_n Y$, where only the last $b_n$ coordinates can be non-zero, and are given by
\begin{equation*}
    \begin{cases}
        \psi_{n, i, u}(z) = 0, &1\leq i \leq u -1, \\
        \psi_{n, i, u}(z) = \frac{\big(-\log(z)\big)^{i-u}}{(i-u)!}, &u \leq i \leq {b_n} .
    \end{cases}
\end{equation*}

The associated remainders are
\begin{equation}
\label{expressiondestaubeta}
    \tau_{n, u}(z) := \prescript{t}{}{\mathcal{S}_n}\cdot\Psi_{n, u} = \sum_{i = u}^{{b_n} } z^{rn}P_{n, i}(z)\frac{\big(-\log(z)\big)^{i-u}}{(i-u)!}. 
\end{equation}

\begin{lemma}
    The family $\big(\tau_{n, u}(z)\big)_{1\leq u\leq {b_n} }$ is linearly independent over $\C$.
\end{lemma}

\begin{proof}
    Assume by contradiction that there exist $\nu_1, ..., \nu_{{b_n} }\in\C$ not all zero such that $\sum_{u=1}^{{b_n} } \nu_u \tau_{n, u} (z) \break = 0$. Let $u_{\min}$ be the smallest index in $\llbracket 1, {b_n} \rrbracket$ such that $\nu_{u_{\min}} \neq 0$.

    We then have the following $\C[z]$-linear relation among the powers of $\log(z)$:
    \begin{equation*}
        \sum_{u=u_{\min}}^{{b_n} } \sum_{i=u}^{{b_n} } \nu_u z^{rn}P_{n, i}(z)\frac{\Big(-\log(z)\Big)^{i-u}}{(i-u)!} = 0.
    \end{equation*}
    Since $\log(z)$ is transcendental over $\C[z]$, the coefficient of $\Big(-\log(z)\Big)^{{b_n} -u_{\min}}$ in this relation has to be zero:
    \begin{equation*}
        \frac{\nu_{u_{\min}}}{({b_n} -u_{\min})!}z^{rn}P_{n, {b_n} }(z) = 0.
    \end{equation*}
    This is a contradiction, since $P_{n, b_n} \neq 0$ by the definition \eqref{définitiondebn} of $b_n$, and $\nu_{u_{\min}} \neq 0$ by definition of $u_{\min}$.
\end{proof}

\subsection{Application of ``Shidlovskii's lemma" and contradiction}
\label{appplicationdeShidlovskii}

We will apply \sref{Theorem}{Shidlovskii} with $d = (h+1)N + b_n$ the size of the matrix $\mathscr{A}_n$ and $\Delta = (r+1)n$ bounding the degrees of the polynomials $S_{n, i}$. 

We consider the finite set of points $\Sigma = \{\mu^\ell z_1 \;|\; 0\leq\ell\leq N-1\}\cup\{0, 1, \infty\}$. Recall that the notion of order was defined in \sref{Subsection}{paragrapheénoncédeShidlovskii}, and that logarithmic factors have no influence on it. From equations \eqref{annulationdesfp} and \eqref{definitiondesrhoq}, we get
\begin{equation*}
    \forall q\in\llbracket 0, h \rrbracket, \;\;\;\forall \ell\in\llbracket 0, N-1\rrbracket, \hspace{1.5cm} \mathrm{ord}_{\mu^{-\ell} z_1}\Big(\rho_{n, q}(\mu^\ell z)\Big) \geq \kappa n.
\end{equation*}

From equation \eqref{expressiondestaubeta} we deduce that
\begin{equation*}
    \forall u\in\llbracket 1, {b_n} \rrbracket, \hspace{1.5cm} \mathrm{ord}_0\Big(\tau_{n, u}(z)\Big) \geq rn.
\end{equation*}

To get a lower bound on the order of $\tau_{n, u}$ at $\infty$, let us rewrite \eqref{expressiondestaubeta} as
\begin{equation*}
    \tau_{n, u}(z) = \sum_{i=u}^{{b_n} } \left(\frac{1}{z}\right)^{-(r+1)n}\frac{1}{z^n}P_{n, i}(z)\frac{\log\left(\frac{1}{z}\right)^{i-u}}{(i-u)!}.
\end{equation*}
Since $\deg P_{n, i} \leq n$, the Laurent polynomials $\frac{1}{z^n}P_{n, i}(z)$ only have negative powers of $z$, and are thus holomorphic at infinity. We deduce that
\begin{equation*}
    \forall u\in\llbracket 1, {b_n} \rrbracket, \hspace{1.5cm} \mathrm{ord}_\infty\Big(\tau_{n, u}(z)\Big) \geq -(r+1)n.
\end{equation*}

\smallskip
Thus, \sref{Theorem}{Shidlovskii} provides a constant $c_{1}$ which only depends on $\Sigma$ and $\mathscr{A}_n$, such that
\begin{equation*}
    (h+1)N\kappa n + {b_n} rn - {b_n} (r+1)n \leq \big((r+1)n+1\big)\big((h+1)N+{b_n} -{b_n} \big) + c_{1},
\end{equation*}
that is
\begin{equation}
\label{nassezgrand2}
    \Big((h+1)(\kappa-r-1)N -{b_n}\Big)n \leq (h+1)N+c_{1} .
\end{equation}

As explained in the remark below, $c_1$ depends on $\mathscr{A}_n$ but it is bounded independently of $n$. Since $\Sigma$ is fixed, we can thus assert that for $n$ large enough, the term multiplying $n$ on the left-hand side of this inequality cannot be positive. That means $(h+1)(\kappa -r-1)N \leq {b_n}  \leq a$.

This is a contradiction with the assumption $(h+1)(\kappa-r-1)N > a$, and the proof of \sref{Proposition}{propositiondéduitedeshidlovskii} is complete.

\begin{remark*}
    The proof of \sref{Theorem}{Shidlovskii} given by \cite[Theorem 3.1]{fischler2018} induces the expression
        \begin{equation*}
            c_{1} = \sum_{\sigma\in\Sigma}\left((\kappa_\sigma + 1)\frac{\nu(\nu-1)}{2}-(\nu-\#J_\sigma)r_\sigma\right) - \nu(\nu-1)-(\nu-\#J_\infty)
        \end{equation*}
        where $\kappa_\sigma$ and $r_\sigma$ are respectively an upper bound on the rank and a lower bound on the generalized orders of a complete system of solutions of $Y' = \mathscr{A}_n Y$. All of these notions are defined in \cite[p.~154]{fischler2018}. In the context of the present paper, we can set
        \begin{enumerate}
            \item[$\bullet$] $\kappa_\sigma = 0$ for all $\sigma\in\Sigma$. Indeed, all of the singular points of the differential system $Y' = \mathscr{A}_n Y$ are regular, which implies that any solution of this system is in the Nilsson class at every point.
            \item[$\bullet$] $r_\sigma = 0$ for all $\sigma\in\Sigma$. Indeed, the matrix $\mathscr{A}_n$ is strictly lower triangular, so that all exponents of the differential system $Y' = \mathscr{A}_n Y$ are zero (see \cite[§6.2.4]{coddingtoncarlson97}, where what we call exponents here are the eigenvalues of the matrix $\mathbf{S}$).
            \item[$\bullet$] $\nu = d = N(h+1)+b_n$ by considering the differential operator $L$ of order $d$ constructed in the proof of \cite[Corollary 6.1.2]{chapitre1}.
        \end{enumerate}
        We also have $\#\Sigma = N+3$ and $\#J_\infty = {b_n} $, so that
        \begin{equation*}
            c_{1} = (N+1)\frac{\big(N(h+1)+b_n\big)\big(N(h+1)+b_n-1\big)}{2}-N(h+1) \leq (N+1)\frac{\big(N(h+1)+a\big)^2}{2}.
        \end{equation*}
\end{remark*}

\section{Conclusion}
\setcounter{subsection}{1}
\label{partie7}

In this section, we apply the following linear independence criterion, which generalizes \cite[Proposition 1, p.6]{fischler2026} and is proved in \cite[Proposition 7.1.1]{chapitre1}, to the linear combinations $\Lambda_{n, (p, k)}(z_1)$. We then achieve the proof of \sref{Theorem}{théorèmeprincipal}.

\begin{proposition}
\label{critèred'indépendancelinéaire}
    (Linear independence criterion)

    Let $\mathbb{K}$ be a number field. We denote by $\mathbb{K}_\infty$ the field $\R$ if $\mathbb{K}\subseteq \R$ and the field $\C$ otherwise.
    
    Let $s\geq 0$ be an integer. Let $\bm{\zeta} = \prescript{t}{}{(}\zeta_0, ..., \zeta_s) \in \mathbb{K}_\infty^{s+1}$ and $\psi: \mathbb{K}_\infty^{s+1} \to \mathbb{K_\infty}$ be a linear form with coefficients in $\mathbb{K}$ such that $\psi(\bm{\zeta}) \neq 0$.

    Let $\tau > 0$ be a real number and $(Q_n)_n$ be a sequence of real numbers that goes to $+\infty$. Let $\mathcal{N}\subseteq \Z_{\geq 0}$ be an infinite set. Suppose we have, for every $n\in\mathcal{N}$, an integer $K_n \geq 1$ and coefficients $\lambda_{n, i, k} \in \mathcal{O}_{\mathbb{K}}$, $\;0\leq i \leq s, \;0\leq k\leq K_n$ such that: 
    \begin{enumerate}
        \item[$(i)$] the linear combinations are bounded above:
            \begin{equation*}
                \forall k\in\llbracket 0, K_n\rrbracket, \;\;\;\left| \sum_{i=0}^s \lambda_{n, i, k}\zeta_i\right| \stackunder{$\leq$}{$\scriptscriptstyle{n\to+\infty}$} Q_n^{-\tau+o(1)},
            \end{equation*}
        \item[$(ii)$] the houses of the coefficients are bounded above:
            \begin{equation*}
                \forall i\in\llbracket 0, s\rrbracket,\;\;\; \forall k\in\llbracket 0, K_n\rrbracket, \;\;\; \house{\lambda_{n, i, k}} \stackunder{$\leq$}{$\scriptscriptstyle{n\to+\infty}$} Q_n^{1+o(1)},
            \end{equation*}
        \item[$(iii)$] for every $n\in\mathcal{N}$ large enough, we have for all $\bm{x} = \prescript{t}{}{(}x_0, ..., x_s)\in\mathbb{K}^{s+1}$ : 
        \begin{equation*}
            \bigg( \forall k \in \llbracket 0, K_n \rrbracket \;\;\; \sum_{i=0}^s \lambda_{n, i, k}x_i = 0\bigg) \;\Rightarrow \;\;\psi(\bm{x}) = 0.
        \end{equation*}
    \end{enumerate}
    Then
    \begin{equation*}
        \mathrm{dim}_{\mathbb{K}}\mathrm{Span}_{\mathbb{K}}(\zeta_0, ..., \zeta_s) \geq \frac{[\mathbb{K}_\infty: \R]}{[\mathbb{K}:\mathbb{Q}]}(\tau+1).
    \end{equation*}
\end{proposition}


Let us apply this criterion to the linear combinations $\Lambda_{n, (p, k)}(z_1)$, $(p, k)\in\llbracket 0, h\rrbracket \times \llbracket (r+1)n+1, \kappa n\rrbracket$ constructed in \autoref{partie4}. We only consider integers $a\in\Z_{>0}$ that are multiples of $10^4$ and large enough for $3.3\sqrt{a\log(a)/N} < \frac{a}{N}$ to hold. After numerical computations aiming at maximizing the constant $0.26$, we let $r=5.3$, $\kappa = 8.8343$, $\omega = 10$, and $\Omega = \lfloor \rho\sqrt{a\log(a)/N}\rfloor$, $h=\nu a$ with $\rho=3.3$ and $\nu = 0.3946$.

According to \eqref{Lambdacommecldeszeta}, the numbers $\Lambda_{n, (p, k)}(z_1)$ are linear combinations of the coordinates of $\bm{\zeta} = \prescript{t}{}{\big(}{f}(0), f(-1), ..., {f}(1-N), L({f}, 1, 1/z_1), ..., L({f}, a+h, 1/z_1)\big)$. Let us check that they satisfy the assumptions of \sref{Proposition}{critèred'indépendancelinéaire}.

According to \sref{Proposition}{majorationdescoefficients}, these linear combinations have coefficients in $\mathcal{O}_{\Q(z_1)}$. Moreover, Propositions\sref{\hspace{1mm}}{majorationdescoefficients} and\sref{\hspace{1mm}}{majorationdescombinaisons} show that assumptions $(i)$ and $(ii)$ hold, with $Q_n =\beta^n$ and $\tau = -\frac{\log(\alpha)}{\log(\beta)}$, so that $Q_n^{-\tau} = \alpha^n$. We recall that
\begin{equation*}
\begin{split}
    \alpha = \den(z_1)^{r+1}\Big(e^3(2a+1) \max(1, |z_1|)g(z_1)\Big)^\kappa \frac{\xi}{r^\Omega}, \\
    \beta = \bigg(\den(z_1) \max\Big(1, \house{z_1}\Big)\bigg)^{r+1} \Big(32e^3(2a+1)\Big)^\kappa \xi,
\end{split}
\end{equation*}
and
\begin{equation*}
    \xi := \exp\Bigg(\frac{\omega\log(2)+2\omega^2 + \omega^2\log(a+1)+\frac{1}{2}\Omega^2\log(r)}{\frac{a}{N}-\omega}\Bigg).
\end{equation*}

Lastly, since $(h+1)(\kappa-r-1)N > h(\kappa-r-1) = 1.00003a > a$, \sref{Proposition}{propositiondéduitedeshidlovskii} shows that assumption $(iii)$ holds for all of the linear forms
\begin{equation*}
    \psi_m (x_0^{<0>}, ..., x_0^{<N-1>}, x_1, ..., x_{a+h}) := x_0^{<m>} \hspace{1cm} 0\leq m \leq N-1.
\end{equation*}
Since $f\neq 0$, we can choose an $m$ that satisfies $\psi_m(\bm{\zeta}) = f(-m) \neq 0$.

We can at last apply \sref{Proposition}{critèred'indépendancelinéaire} with $\K = \Q(z_1)$ and the linear form $\psi_{m}$ as above. We obtain
\begin{equation*}
    \mathrm{dim}_{\Q(z_1)}\mathrm{Span}_{\Q(z_1)} \Big\{{f}(0), ..., {f}(N-1), L({f}, i, 1/z_1) \;|\;1\leq i\leq a+h\Big\} \stackunder{$\geq$}{$\scriptscriptstyle{a\to+\infty}$} \hspace{5mm} \frac{[\K_\infty: \R]}{[\Q(z_1): \Q]}\left( 1 - \frac{\log(\alpha)}{\log(\beta)}\right).
\end{equation*}

\smallskip
Let us compute the asymptotics of this lower bound as $a \xrightarrow[]{} +\infty$:
\begin{equation*}
    \log(\xi) \stackunder{$=$}{$\scriptscriptstyle{a\to+\infty}$} \frac{\frac{\rho^2}{2N}a\log(a)\log(r)}{\frac{a}{N}}\big(1+o(1)\big) \stackunder{$=$}{$\scriptscriptstyle{a\to+\infty}$} \frac{\rho^2\log(r)}{2}\big(1+o(1)\big) \log(a),
\end{equation*}
then
\begin{equation*}
    \log(\beta) \stackunder{$=$}{$\scriptscriptstyle{a\to+\infty}$} \big(\kappa\log(a) + \log(\xi)\big)\big(1+o(1)\big) \stackunder{$=$}{$\scriptscriptstyle{a\to+\infty}$} \bigg(\kappa + \frac{\rho^2\log(r)}{2}\bigg)\big(1+o(1)\big)\log(a)
\end{equation*}
and
\begin{align*}
    -\log(\alpha) &\stackunder{$=$}{$\scriptscriptstyle{a\to+\infty}$} \left(\frac{\rho}{\sqrt{N}}\sqrt{a\log(a)}+O(1)\right)\log(r) - \kappa \log(a)\big(1+o(1)\big) - \log(\xi)\\
    &\stackunder{$=$}{$\scriptscriptstyle{a\to+\infty}$} \frac{\rho}{\sqrt{N}}\log(r)\big(1+o(1)\big)\sqrt{a\log(a)}.
\end{align*}
Hence,
\begin{equation*}
     - \frac{\log(\alpha)}{\log(\beta)} \stackunder{$=$}{$\scriptscriptstyle{a\to+\infty}$} \frac{2\rho\log(r)\big(1+o(1)\big)}{\big(2\kappa + \rho^2\log(r)\big)\sqrt{N}}\sqrt{\frac{a}{\log(a)}}.
\end{equation*}

\smallskip
The numbers ${f}(-m), 1\leq m\leq N,\;$ are in finite number: they span a $\Q(z_1)$-vector space of dimension $O(1) = o\left(\sqrt{\frac{a}{\log(a)}}\right)$ as $a \to +\infty$ and can thus be ignored.

By noticing that $$\sqrt{\frac{a}{\log(a)}} \stackunder{$=$}{$\scriptscriptstyle{a\to+\infty}$} \frac{1}{\sqrt{1+\nu}}\big(1+o(1)\big)\sqrt{\frac{a+h}{\log(a+h)}},$$ we can write
\begin{align*}
    \mathrm{dim}_{\Q(z_1)}\mathrm{Span}_{\Q(z_1)} \Big\{L({f}, i, 1/z_1) \;|\;&1\leq i\leq a+h \Big\} \\
    &\stackunder{$\geq$}{$\scriptscriptstyle{a\to+\infty}$} \frac{2\rho\log(r)\big(1+o(1)\big)}{\big(2\kappa + \rho^2\log(r)\big)\sqrt{1+\nu}}\frac{[\K_\infty: \R]}{\sqrt{N} [\Q(z_1): \Q]}\sqrt{\frac{a+h}{\log(a+h)}}.
\end{align*}

Now, with the choice of numerical values described at the beginning of this section, we compute $\frac{2\rho\log(r)}{\big(2\kappa + \rho^2\log(r)\big)\sqrt{1+\nu}} = 0.26013... > 0.26$. Notice that $[\K_\infty:\R] = 1$ if $\Q(z_1)\in\R$, and $2$ otherwise.

By letting $s = a+h$ and recalling that $z_1 = 1/z_0$, and thus $\Q(z_1) = \Q(z_0)$, we conclude the proof of \autoref{théorèmeprincipal} for integers $s$ large enough and written as $1.3946a$, with $10^4 | a$.

Finally, for arbitrary large enough integers $s$, we pick the largest multiple $a_s$ of $10^4$ satisfying $1.3946a_s \leq s$. We have $$\mathrm{Span}_{\Q(z_1)} \Big\{L({f}, i, 1/z_1) \;|\;1\leq i\leq s \Big\} \supseteq \mathrm{Span}_{\Q(z_1)} \Big\{L({f}, i, 1/z_1) \;|\;1\leq i\leq 1.3946a_s \Big\},$$ and since $\frac{1.3946a_s}{s} \xrightarrow[s\to+\infty]{} 1$ and the actual constant is $0.26013... > 0.26$, the result holds for all sufficiently large integers $s$.


\section*{AI utilization notice}

Throughout the writing of this paper, generative AI was used for bibliographic purposes and language checks.

\smallskip
\bibliographystyle{alpha}
\bibliography{bibliography}

\begin{thebibliography}{Chu84}

\bibitem[Ap{\'e}79]{apery79}
R.~Ap{\'e}ry.
\newblock Irrationnalité de $\zeta(2)$ et $\zeta(3)$.
\newblock {\em Astérisque}, 61:11--13, 1979.

\bibitem[BB85]{bertrandbeukers85}
D.~Bertrand and F.~Beukers.
\newblock Équations différentielles linéaires et majorations de multiplicités.
\newblock {\em Ann. Sci. École Norm. Sup.}, 18.1:181--192, 1985.

\bibitem[Ber12]{bertrand2012}
D.~Bertrand.
\newblock Le théorème de {S}iegel-{S}hidlovsky revisité.
\newblock In {\em Number theory, Analysis and Geometry : in memory of Serge Lang}, Astérisque, pages 51--67. Springer, 2012.

\bibitem[BR01]{ballrivoal2001}
K.~Ball and T.~Rivoal.
\newblock Irrationalité d'une infinité de valeurs de la fonction zêta aux entiers impairs.
\newblock {\em Invent. Math.}, 146:193--207, 2001.

\bibitem[CC85]{chudnovskychudnovsky85}
D.~Chudnovsky and G.~Chudnovsky.
\newblock Applications of {P}adé approximations to {D}iophantine inequalities in values of {G}-functions.
\newblock {\em Number theory (New York 1983/84), Lecture Notes in Math., Springer}, 1135:9--51, 1985.

\bibitem[CC97]{coddingtoncarlson97}
E.~A. Coddington and R.~Carlson.
\newblock {\em Linear {O}rdinary {D}ifferential {E}quations}.
\newblock Soc. for Industrial and Applied Math., 1997.

\bibitem[Chu84]{chudnovsky84}
G.~Chudnovsky.
\newblock On applications of {D}iophantine approximations.
\newblock {\em Proceedings of the National Academy of Sciences of the United States of America}, 81.22:7261--7265, 1984.

\bibitem[Fis04]{fischler2004}
S.~Fischler.
\newblock Irrationalité de valeurs de zêta d'après {A}péry, {R}ivoal, ...
\newblock In {\em Sém. Bourbaki 2002/03}, number 294 in Astérisque, pages 27--62, 2004.
\newblock exp. no. 910.

\bibitem[Fis18]{fischler2018}
S.~Fischler.
\newblock Shidlovsky's multiplicity estimate and irrationality of zeta values.
\newblock {\em J. Austral. Math. Soc.}, 105.2:147--172, 2018.

\bibitem[Fis20]{fischler2020}
S.~Fischler.
\newblock Irrationality of values of ${L}$-functions of {D}irichlet characters.
\newblock {\em J. London Math. Soc.}, 101.2:857--876, 2020.

\bibitem[Fis26]{fischler2026}
S.~Fischler.
\newblock Linear independence of odd zeta values using {S}iegel’s lemma.
\newblock {\em J. London Math. Soc.}, 113.4, 2026.

\bibitem[FR03]{fischlerrivoal2003}
S.~Fischler and T.~Rivoal.
\newblock Approximants de {P}adé et séries hypergéometriques équilibrées.
\newblock {\em Journal de Math. Pures et Appliquées}, 82.10:1369--1394, 2003.

\bibitem[FR20]{fischlerrivoal2020}
S.~Fischler and T.~Rivoal.
\newblock Linear independence of values of {G}-functions.
\newblock {\em J. Eur. Math. Soc.}, 22.5:1531--1576, 2020.

\bibitem[FR21]{fischlerrivoal2021}
S.~Fischler and T.~Rivoal.
\newblock Linear independence of values of {G}-functions, {I}{I}. {O}utside the disk of convergence.
\newblock {\em Ann. Math. Québec}, 45.1:53--93, 2021.

\bibitem[FSZ19]{fischlersprangzudilin2019}
S.~Fischler, J.~Sprang, and W.~Zudilin.
\newblock Many odd zeta values are irrational.
\newblock {\em Compositio Math.}, 155:938--952, 2019.

\bibitem[Hat90]{hata90}
M.~Hata.
\newblock On the linear independence of the values of polylogarithmic functions.
\newblock {\em J. Math. Pures Appl.}, 69.2:133--173, 1990.

\bibitem[Lai25]{lai2025}
L.~Lai.
\newblock A note on the number of irrational odd zeta values, {II}.
\newblock {\em Preprint arXiv:2501.05321 [math.NT]}, 2025.

\bibitem[LY20]{laiyu2020}
L.~Lai and P.~Yu.
\newblock A note on the number of irrational odd zeta values.
\newblock {\em Compositio Math.}, 156:1699--1717, 2020.

\bibitem[Mar06]{marcovecchio2006}
R.~Marcovecchio.
\newblock Linear independence of linear forms in polylogarithms.
\newblock {\em Annali Scuola Norm. Sup. Pisa}, V.1:1--11, 2006.

\bibitem[Mis26]{chapitre1}
L.~Mistiaen.
\newblock Linear independence of values of {D}irichlet {L} functions.
\newblock {\em Preprint arXiv:2512.02612 [math.NT]}, 2026.

\bibitem[Nes85]{nesterenko85}
Y.~V. Nesterenko.
\newblock Linear independence of numbers.
\newblock {\em Vestnik Moskov. Univ. Ser. I Mat. Mekh.}, 1:46--49, 1985.

\bibitem[Nes96]{nesterenko96}
Y.~V. Nesterenko.
\newblock A few remarks on $\zeta(3)$.
\newblock {\em Math. Notes}, 59.6:625--636, 1996.

\bibitem[Neu99]{neukirch97}
J.~Neukirch.
\newblock {\em Algebraic {N}umber {T}heory}.
\newblock Springer, 1999.

\bibitem[Nik79]{nikishin79}
E.~M. Nikishin.
\newblock On the irrationality of the values of the functions ${F}(x, s)$.
\newblock {\em Mat. Sbornik}, 37.3:381--388, 1979.

\bibitem[Nis11]{nishimoto2011}
M.~Nishimoto.
\newblock On the linear independence of the special values of a {D}irichlet series with periodic coefficients.
\newblock {\em Preprint arXiv:1102.3247 [math.NT]}, 2011.

\bibitem[Oes93]{oesterle93}
J.~Oesterlé.
\newblock Polylogarithmes.
\newblock In {\em Séminaire {N}. {B}ourbaki 1992/93}, number 216 in Astérisque, pages 49--67, 1993.
\newblock exp. no. 762.

\bibitem[Riv00]{rivoal2000}
T.~Rivoal.
\newblock La fonction {Z}êta de {R}iemann prend une infinité de valeurs irrationnelles aux entiers impairs.
\newblock {\em C.R.A.S. Paris, Sér. I Math.}, 331.4:267--270, 2000.

\bibitem[Riv03]{rivoal2003}
T.~Rivoal.
\newblock Indépendance linéaire des valeurs des polylogarithmes.
\newblock {\em J. Théor. Nombres Bordeaux}, 15.2:551--559, 2003.

\bibitem[Shi89]{shidlovskii89}
A.~B. Shidlovskii.
\newblock {\em Transcendental numbers}.
\newblock de Gruyter Studies in Math, no. 12, 1989.

\end{thebibliography}

\end{document}